\documentclass[11pt]{article}
\usepackage{enumerate,latexsym}
\usepackage{latexsym}
\usepackage{amsmath,amssymb}
\usepackage{times}

\usepackage[ansinew]{inputenc}
\usepackage{graphicx}
\usepackage{color}

\usepackage{amsthm}

\def\rth{\mathbb{R}^3}
\newcommand{\R}{\mathbb{R}}
\def\g{\gamma}

\def\B{\mathbb{B}}
\def\N{\mathbb{N}}

\newcommand{\nc}{\newcommand}
\newcommand{\ben}{\begin{enumerate}}
\newcommand{\bit}{\begin{itemize}}
\newcommand{\een}{\end{enumerate}}
\newcommand{\eit}{\end{itemize}}
\newcommand{\wh}{\widehat}
\newcommand{\Int}{\mathrm{Int}}

\newcommand{\cD}{{\mathcal D}}
\newcommand{\cC}{{\mathcal C}}

\newcommand{\cR}{{\mathcal R}}

\newcommand{\cH}{{\mathcal H}}
\newcommand{\cF}{{\mathcal F}}

\newcommand{\wt}{\widetilde}

\newcommand{\cN}{\mathcal N}
\newcommand{\ov}{\overline}

\def\a{{\alpha}}

\def\t{{\theta}}

\def\g{{\gamma}}
\def\G{{\Gamma}}

\def\de{{\delta}}
\def\be{{\beta}}
\def\ve{{\varepsilon}}
\def\S{\Sigma}

\def\centerbmp#1#2#3{\vskip#2\relax\centerline{\hbox to#1{\special
    {bmp:#3 x=#1, y=#2}\hfil}}}

\newtheorem{theorem}{Theorem}[section]
\newtheorem{lemma}[theorem]{Lemma}
\newtheorem{proposition}[theorem]{Proposition}

\newtheorem{remark}[theorem]{Remark}
\newtheorem{ddescription}[theorem]{Description}
\newtheorem{property}[theorem]{Property}
\newtheorem{corollary}[theorem]{Corollary}
\newtheorem{definition}[theorem]{Definition}
\newtheorem{conjecture}[theorem]{Conjecture}

\newtheorem{assertion}[theorem]{Assertion}

\newtheorem{claim}[theorem]{Claim}

\newcommand{\ed}{\end{document}}
\nc{\bl}{\begin{lemma} }

\nc{\el}{\end{lemma} }

\nc{\bt}{\begin{theorem} }

\nc{\et}{\end{theorem} }
\newcommand{\rc}{ \renewcommand }

\rc{\v}{    \overset{\longrightarrow} }

\definecolor{b}{rgb}{.1,.1,.7}

\date{}

\begin{document}

\begin{title}
{Curvature estimates for constant mean curvature surfaces}
\end{title}

\begin{author}
{William H. Meeks III\thanks{This material is based upon
   work for the NSF under Award No. DMS-1309236.
   Any opinions, findings, and conclusions or recommendations
   expressed in this publication are those of the authors and do not 
   necessarily reflect the views of the NSF.}
   \and Giuseppe Tinaglia\thanks{The second author was partially
   supported by
EPSRC grant no. EP/M024512/1}}
\end{author}
\maketitle



%
%
%

\begin{abstract}
We derive extrinsic curvature estimates for compact disks embedded
in $\rth$ with nonzero constant mean curvature.

\vspace{.3cm}


\end{abstract}
\maketitle


\section{Introduction.}

For clarity of  exposition, we will call an oriented surface
$M$ immersed in $\rth$ an {\it $H$-surface} if it
is {\it embedded}, {\em connected}  and it
has {\it positive constant mean curvature $H$}. We will  call an
$H$-surface an {\em $H$-disk} if the $H$-surface is homeomorphic
to a closed  disk in the Euclidean plane.

The main result in this paper is the following curvature estimate for compact disks embedded
in $\rth$ with nonzero constant mean curvature.

\begin{theorem}[Extrinsic Curvature Estimates] \label{ext:cest}
Given $\delta,\cH>0$, there exists a constant $K_0(\delta,\cH)$ such that
for any $H$-disk $\cD$ with $H\geq \cH$,
$${\large{\LARGE \sup}_{\large \{p\in \cD \, \mid \, d_{\rth}(p,\partial
\cD)\geq \delta\}} |A_\cD|\leq  K_0(\delta,\cH)}.$$
\end{theorem}

We wish to emphasize  to the reader that the curvature estimates for
embedded constant mean curvature disks given in
Theorem~\ref{ext:cest}  depend {only} on the {\em lower} positive bound
$\cH$ for their mean curvature. Previous important examples
of curvature estimates for constant mean curvature surfaces, assuming certain geometric conditions,
can be found in the literature;
see for instance~\cite{boti1,boti3, cs1,cm23,cm35,rst1,sc3,ss1,ssy1, tin3,tin4}.

We now give a brief outline of our approach to proving   Theorem~\ref{ext:cest}.
The proof of this theorem is by contradiction and relies on an accurate
geometric description of a $1$-disk near interior points where
the norm of the second fundamental form becomes
arbitrarily large. This geometric description
is inspired by the pioneering work of Colding and Minicozzi
in the minimal case~\cite{cm21,cm22,cm24, cm23};
however in the constant positive mean curvature setting this
description leads to curvature estimates.
Rescalings of a helicoid
give rise to a sequence of embedded minimal disks with arbitrarily large
norms of their second fundamental forms at points that can be arbitrarily far
from their boundary curves; therefore in the minimal setting,  curvature estimates
do not hold.

Finally, this curvature estimate is the key to proving several fundamental results
about the geometry of $H$-surfaces, see~\cite{mt12,mt13,mt14,mt9}. In turn, these
results are used in~\cite{mt15} to generalize the extrinsic curvature estimate to an intrinsic one.

The  theory developed
in this manuscript also provides key tools for understanding the geometry
of $H$-disks in a Riemannian three-manifold, especially in the case that
the manifold is locally homogeneous.  These generalizations and applications of the work presented here
will appear in our forthcoming paper~\cite{mt1}. 


\section{An Extrinsic Curvature Estimate for certain planar domains} \label{sec:out}
First,  we fix some notations that we   use throughout the paper. \vspace{.1cm}

\bit
\item For $r>0$ and $p\in \rth$, $\B(p,r):
=\{x=(x_1,x_2,x_3)\in\rth\mid |p-x|<r\}$ and $\B(r):=\B(\vec{0},r)$.
\item For $r>0$ and $p\in \S$, a surface in $\rth$, $B_{\S}(p,r)$ denotes
the open intrinsic ball in $\S$ of radius $r$.
\item For
positive numbers $r$, $h$ and $t$,
\[
C(r,h,t):=\{(x_1,x_2,x_3)\in\rth\mid (x_1-t)^2+x_2^2\leq r^2, \,|x_3|\leq h \},
\]
which is the solid closed
vertical cylinder of radius $r$, height $2h$ and centered at the
point $(t,0,0)$: $$C(r,h):=C(r,h,\vec0).$$
\item For positive numbers $r_1>r_2>0$, we let
\[
A(r_1, r_2):=\{(x_1,x_2,x_3)\in\rth\mid  r_2<\sqrt{x_1^2+x_2^2}<r_1, \,x_3=0\},
\]
which is the annulus in the plane $ \{x_3=0\}$, centered at the
origin with outer radius $r_1$ and inner radius $r_2$.
\item  For $R>0$ and $p\in \rth$, $C_R(p)$ denotes the infinite solid vertical cylinder
centered at $p$ of radius $R$ and $C_R:=C_R(\vec 0)$.
\eit  \vspace{.1cm}

The first step in proving the intrinsic curvature
estimate for $H$-disks in Theorem~\ref{cest} is to obtain an
extrinsic  curvature estimate, Theorem~\ref{thm3.1} below,
for certain compact  $H$-surfaces
that are planar domains.

Before stating
Theorem~\ref{thm3.1}, we describe the notion of the
flux of a 1-cycle in an $H$-surface; see for instance~\cite{kks1,ku2,smyt1}
for further discussion of this invariant.

\begin{definition} \label{def:flux} {\em
Let $\gamma$ be a piecewise-smooth 1-cycle in an $H$-surface $M$. The  flux of
$\gamma$ is $\int_{\gamma}(H\gamma+\xi)\times \dot{\gamma}$, where $\xi$
is the unit normal to $M$ along $\gamma$ and $\g$ is parameterized by arc length. }
\end{definition}

The flux is a homological invariant and
we say that $M$ has {\em  zero flux} if the flux of any 1-cycle in $M$ is zero;
in particular, since the first homology group of a disk is
zero, the flux of an $H$-disk is zero.

\begin{theorem} \label{thm3.1}
Given $\ve>0$, $m\in \N$  and $H \in(0, \frac{1}{2\ve})$, there exists
a constant $K(m,\ve, H)$ such that the following holds.  Let
$M\subset \ov{\B}(  \ve  )$ be a compact, connected $H$-surface of genus zero
with at most $m$ boundary components, $\vec{0}\in M$, $\partial M \subset \partial
\B( \ve )$ and $M$ has zero flux. Then:
$$|A_M| (\vec{0})\leq  K(m,\ve, H).$$
\end{theorem}

\begin{remark} \label{rem:2.3} {\em In Proposition~\ref{number}
we prove that given an $H$-disk $\Sigma$
such that $\partial \Sigma\subset \partial \B(\ve)$ with $H \in(0, \frac{1}{2\ve})$,
then  the number of boundary components of a connected
component of $\Sigma\cap\B(\ve)$ is bounded from
above by some natural number $N_0$ that is independent of $\Sigma$. Therefore,
Theorem~\ref{thm3.1} together with Proposition~\ref{number} gives
the extrinsic curvature estimate constant $K(N_0,\ve, H)$ for $\S$.}
\end{remark}

The proof of Theorem~\ref{thm3.1} is quite long and
involves the proofs of other important results. Before entering into the
details we outline its organization.

We first  introduce the notion of multi-valued graph, see~\cite{cm22}
for further discussion and Figure~\ref{3-valuedgraph}.
\begin{figure}[h]
\begin{center}
\includegraphics[width=11cm]{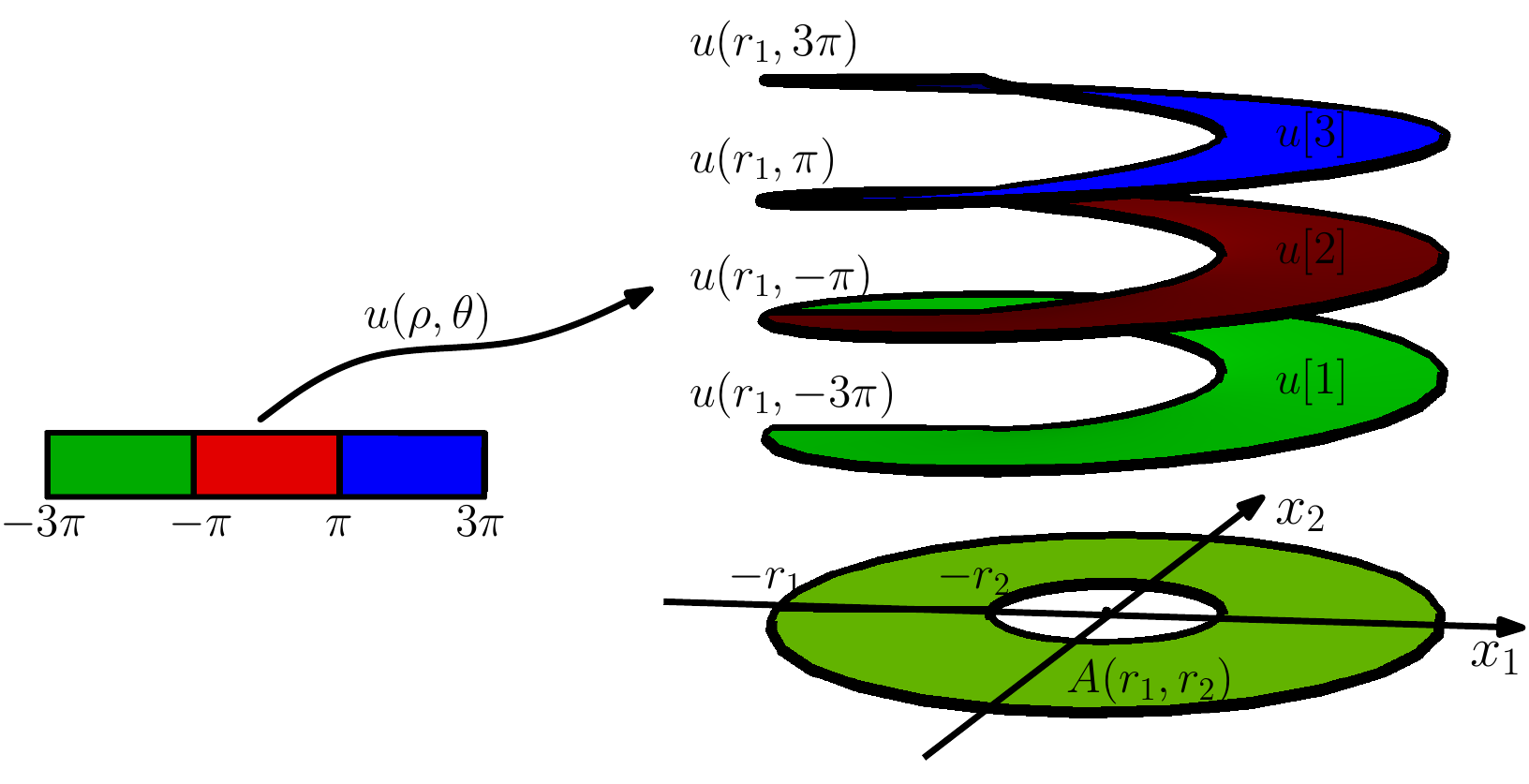}
\caption{A right-handed 3-valued graph.}
\label{3-valuedgraph}
\end{center}
\end{figure}
Intuitively, an $N$-valued graph is a simply-connected embedded surface covering an
annulus such that over a neighborhood of each point of the annulus, the
surface consists of $N$ graphs. The stereotypical  infinite multi-valued
graph is half of a helicoid, i.e., half of an infinite  double-spiral staircase.

\begin{definition}[Multi-valued graph]\label{multigraph} {\rm
Let $\mathcal{P}$ denote the universal cover of the
punctured $(x_1,x_2)$-plane, $\{(x_1,x_2,0)\mid (x_1,x_2)\neq (0,0)\}$, with global coordinates
$(\rho , \theta)$.
\ben[1.] \item
An {\em $N$-valued graph over the annulus $ A(r_1,r_2)$} is a single valued graph
$u(\rho, \theta)$ over $\{(\rho ,\theta )\mid r_2\leq \rho
\leq r_1,\;|\theta |\leq N\pi \}\subset \mathcal{P}$, if $N$ is odd,
or over $\{(\rho ,\theta )\mid r_2\leq \rho
\leq r_1,\;(-N+1)\pi\leq \theta \leq \pi (N+1)\}\subset \mathcal{P}$, if $N$ is even.
\item An  $N$-valued graph $u(\rho,\t)$ over the annulus $ A(r_1,r_2)$ is
called {\em righthanded} \, [{\em lefthanded}] if whenever it makes sense,
$u(\rho,\t)<u(\rho,\t +2\pi)$ \, [$u(\rho,\t)>u(\rho,\t +2\pi)$]
\item The set $\{(r_2,\theta, u(r_2,\theta)), \theta\in[-N\pi,N\pi]\}$ when $N$
is odd (or $\{(r_2,\theta, u(r_2,\theta)), \theta\in[(-N+1)\pi,(N+1)\pi]\}$ when
$N$ is even) is the {\em inner boundary} of the $N$-valued graph.
\een }
\end{definition}

Note that the boundary of  an $N$-valued graph
consists of four connected smooth arcs. They are
a spiral on $\partial C_{r_1} $,  a spiral on $\partial C_{r_2} $ which is the inner boundary,
and two arcs $\gamma^\pm$ connecting the top
and bottom endpoints of these spirals. These latter arcs are of the form
\[
\gamma^\pm=\{(t,0,\phi^\pm(t))\mid t\in [-r_1,-r_2]\},
\]
where  $\phi^\pm$ are smooth functions.

For simplicity, in the next Definitions~\ref{def:middle}
and \ref{def:W}, we  assume $N$ is odd and that the $N$-valued graph is righthanded;
the analogous definitions when $N$ is even or the graph is lefthanded are
left to the reader.  When we encounter
$N$-valued graphs in the proof of Theorem~\ref{thm3.1}, we will also assume, without loss
of generality, that $N$ is odd and the  $N$-valued graph  is righthanded.

\begin{definition} \label{def:middle} {\em
We call the set
\[
u[k]:=\{(\rho,\theta,u(\rho, \theta)) \mid r_2\leq \rho
\leq r_1,\; (-N +2k-2)\pi \leq \theta \leq (-N +2k)\pi \},
\]
 where
$k=1,\dots, N$, the {\em $k$-th sheet of the $N$-valued graph}
and $u_{\rm mid}:=u[[N\slash 2]+1]$ is its {\em middle sheet}; here,
$[N\slash2]$ denotes the integer part of $N\slash 2$.}
\end{definition}

\begin{definition}\label{def:W} {\em
Given an $N$-valued graph $u$, $N>1$, over the annulus $ A(r_1,r_2)$
we let $W[u]$ denote the open solid {\em region
trapped between the sheets} of $u$. Namely, $W[u]$ is the
connected, open, bounded solid region of $\rth$ whose boundary consists of
the $N$-valued graph $u$ together with the following union
of vertical segments: the set of vertical segments whose
end points are $(r_2,\theta, u(r_2,\theta))$,
$( r_2,\theta+2\pi, u(r_2,\theta+2\pi))$ with $\theta\in[-N\pi,(N-2)\pi]$,
the set of vertical segments whose end points are
$(r_1,\theta, u(r_1,\theta))$, $(r_1,\theta+2\pi, u(r_1,\theta+2\pi))$
with $\theta\in[-N\pi,(N-2)\pi]$, the set of vertical
segments whose end points are $(\rho,(N-2)\pi, u(\rho,(N-2)\pi))$,
$(\rho,N\pi, u(\rho,N\pi))$ with $\rho\in[r_2,r_1]$
and the set of vertical segments whose end points are
$(\rho,-N\pi, u(\rho,-N\pi))$,
$(\rho,-(N-2)\pi, u(\rho,-(N-2)\pi))$ with $\rho\in[r_2,r_1]$.
We can parameterize
the set $W[u]$ in a natural way by using coordinates $(\rho,\theta,x_3)$ with
$(\rho,\theta,x_3)$ in an open subset of $  (r_2,r_1)\times (-N\pi,(N-2)\pi)\times \R$.}
\end{definition}

Below are the main steps of the proof of Theorem~\ref{thm3.1}.
Arguing by contradiction, suppose that the theorem fails. In this
case, for some $\ve>0$ and $H\in (0,\frac 1{2\ve})$ there exists a sequence
$M_{n}$ of $H$-surfaces satisfying the hypotheses of the theorem
and $|A_{M_{n}}|(\vec{0})>n$. Since $H$
is fixed, by scaling and after reindexing the elements of the
sequence, we can assume that $\ve<\frac12$,
$H=1$ and $|A_{M_{n}}|(\vec{0})>n$. After replacing $M_{n}$ by
a subsequence  composed by translations, with translation vectors $\vec v_n$
having $|\vec v_n|<\ve\slash 4$, and a fixed rotation,
that we shall still call $M_{n}$, when $n$ is sufficiently large
we will prove the following.
\ben[1.]\item In Section~\ref{sec3}, we show that  $M_{n}$ is closely
approximated  by a vertical helicoid on a small scale
around the origin.
\item In Section~\ref{findmvg}, we prove
that there exists a sequence of embedded stable
\underline{minimal} disks $E(n)$
on the mean convex side of $M_{n}$, where $E(n)$ contains a
multi-valued graph $E^g_n$ that starts near the origin and extends horizontally on a scale
proportional to $\ve$.
\item In Section~\ref{sec3.3}, we use the existence of the minimal
multi-valued graph $E^g_n$ to prove that $M_{n}$ contains a pair
$G_n^{up},G_n^{down}$ of oppositely oriented 3-valued graphs
with norms of the gradients bounded by 1 that start near the origin
and  extend horizontally on a scale proportional to $\ve$, and
satisfying $\ov{W}[G_n^{up}]\cap G_n^{down}$ is a 2-valued graph.
\item Finally, we use the existence of the pairs of 3-valued graphs
$G_n^{up},G_n^{down} \subset M_n$  to obtain a contradiction and thereby
prove the curvature estimate described in Theorem~\ref{thm3.1}. \een

\subsection{Local picture near a point of large curvature.} \label{sec3}
In this section we describe the geometry of constant mean curvature
planar domains with zero flux near  interior points of large
curvature. Roughly speaking, nearby a point with large
norm of the second fundamental form, the planar domain contains
a pair of oppositely oriented multi-valued graphs like in a helicoid. In the
case of embedded minimal disks  such a description was
given by Colding and Minicozzi in~\cite{cm22};
see also~\cite{tin1,tin2} for related results.
By rescaling arguments this description
can be improved upon once one knows that the helicoid is the unique complete,
embedded, non-flat minimal surface in $\rth$ as explained
below;  see~\cite{mr8} and also~\cite{bb1} for
proofs of the uniqueness of
the helicoid
which are based in part on results in \cite{cm21,cm22,cm24,cm23,cm35}.

 Fix $\bar \ve\in (0,\frac \ve 4)$, a sequence $p_n\in M_n\cap \ov{\B}( \bar\ve)$ is a sequence of, so-called,
 points of almost-maximal curvature, if $p_n$ is a maximum for the function
\[
f_n\colon M_n\cap\ov{\B}(\bar\ve) \to [0, \infty), \quad f_n(\cdot)=|A_{M_n}|(\cdot)(\bar\ve-|\cdot|);
\] note that the points $p_n$ lie in the interior $\B(\ov{\ve})$ of $\ov{\B}(\ov{\ve})$.

By a standard compactness argument, see for instance
the proof and statement of Theorem~1.1 in~\cite{mpr20}
or Lemma~5.5 in~\cite{mr8}, given a sequence
$p_n\in M_n\cap \B( \bar\ve)$ of  points of almost-maximal
curvature, there exist positive numbers $\delta_n$,
with $\delta_n\to 0$, and  a subsequence, that we still call $M_{n}$, such that the surfaces
$\widehat{M}_{n}=M_{n}\cap\B(p_n,\delta_n)$, satisfy:
 \ben[1.]\item $\lim_{n\to \infty} \delta_n \cdot |A_{M_{n}}|(p_n)=\infty$.
\item $\sup_{p\in \widehat{M}_{n}}|A_{\widehat{M}_{n}}|(p)\leq
(1+\frac{1}{n})\cdot|A_{M_{n}}|(p_n).$
\item The sequence of translated and rescaled
surfaces
$$\Sigma_{n}=\frac{1}{\sqrt 2}|A_{M_{n}}|(p_n)\cdot(\widehat{M}_{n}-p_n)$$
converges smoothly with multiplicity one or two on compact subsets of
$\rth$ to a connected, properly embedded, nonflat minimal surface
$\Sigma_{\infty}$ with bounded norm of the second fundamental form. More precisely,
\[
\sup_{\Sigma_\infty}|A_{\Sigma_{\infty}}|\leq |A_{\Sigma_{\infty}}|(\vec{0})=\sqrt 2.
\]
\item In the case that the convergence has multiplicity two, then the mean curvature
vectors of the two surfaces limiting to $\Sigma_\infty$
point away from the collapsing region  between them;
recall that the mean curvature vector is $H\xi$ where $\xi$ is the unit normal.
\item Given  any smooth  loop $\a$ in $\Sigma_{\infty}$, for each $n$ sufficiently
large, $\a$ has a normal lift $\a_n \subset \S_{n}$ such that the lifted loops
converge smoothly with multiplicity one to $\a$ as $n\to \infty$;
in the case the convergence has multiplicity two,
there are
exactly two such pairwise disjoint  normal lifts of $\a_n$ to  $\S_n$.
\een

We now give some details on obtaining the above description.
Regarding the convergence of the surfaces
$\S_n$ to $\S_\infty$, the fact that a subsequence
of the rescaled surfaces $\S_n$ converges smoothly to
a connected, properly embedded, nonflat minimal surface
$\Sigma_{\infty}$ with bounded norm of the second fundamental
form can be seen as follows. A standard compactness argument shows that
a subsequence converges to a nonflat minimal lamination of $\rth$ with leaves having uniformly
bounded norms of their second fundamental forms; see for  instance the proof of Lemma~5.5
in~\cite{mr8} for this type of compactness argument. Hence by Theorem~1.6 in~\cite{mr8}, the lamination
consists of a connected, properly embedded, nonflat minimal surface $\S_{\infty}$. The convergence
of the $\S_n$ to  $\S_{\infty}$ is with multiplicity at most two
 because otherwise a higher multiplicity of convergence
 would allow the construction of a  positive Jacobi function on $\S_{\infty}$
 and so $\S_{\infty}$ would be a stable minimal surface that is
a plane~\cite{cp1,fs1,po1}, which is false; hence the multiplicity of convergence is
at most two. The construction of this positive Jacobi function follows the
construction of a similar positive Jacobi function on a limit minimal surface in the proof of
Lemma~5.5 in~\cite{mr8}. However, in our setting of multiplicity of convergence greater
than two, one fixes an arbitrary compact domain $\Omega\subset \S_{\infty}$ containing the origin
and, for $n$ sufficiently large,
finds two domains $\Omega_1(n),\Omega_2(n)\subset \S_n$, that are expressed as
small normal oriented graphs
over $\Omega$, and have the same signed small constant mean curvatures.
The proof of Lemma~5.5 then produces
a positive Jacobi function $F_{\Omega}$ on $\Omega$ with $F_{\Omega}(\vec{0})=1$;
since $\Omega$ is arbitrary, it follows that $\S_{\infty}$ would be stable,
which we already showed is impossible.
This same argument explains
 why, in the case the multiplicity of convergence is two, the similar
 geometric properties in item~4 above hold.
For still further details on this type of multiplicity of convergence
at most two argument and for the collapsing description in item~4,
see for example the proof of Case A of Proposition~3.1 in~\cite{mt9}. There
 it is explained, in a similar but slightly more general situation,
 that the convergence of a certain sequence of $H_n$-disks
to a non-flat limit surface, which is  a helicoid, has multiplicity
at most two.

The multiplicity of convergence being at most two,
together with a standard curve lifting argument,  implies
$\Sigma_{\infty}$ has genus zero and that item~5 holds.
Then  the loop lifting property in item~5
 implies that the flux of the limit
properly embedded minimal surface $\Sigma_\infty$ is
zero. Since a minimal surface properly embedded in $\rth$ with more
than one end has nonzero flux~\cite{cmw1}, then $\Sigma_{\infty}$ has one
end. However, a genus zero surface with one end is simply-connected and so,
$\Sigma_{\infty}$ is a helicoid by~\cite{mr8}.

In summary, arbitrarily
close to the origin, depending on the choice of $\bar\ve$, there exist helicoid-like surfaces (the
surfaces $\Sigma_{n}$ above) forming on $M_{n}$.
Without loss of
generality, after a translation  of the $M_{n}$, we may assume
that $p_n=\vec{0}$ and, abusing the notation, we will still assume that
$\partial M_{n}\subset \partial \B(\ve)$. In actuality
$\partial M_n$ lies on the boundary of a translation of $\partial \B(\ve)$.
 The arguments in the following constructions
would either remain the same or can be easily modified,
if one desires to keep track of these translations.

After a possible rotation, we will also assume that $\Sigma_{\infty}$ is a vertical
helicoid containing the vertical $x_3$-axis and the
$x_2$-axis.

The proof of Theorem~\ref{thm3.1} breaks up into the following two cases:
\vspace{.2cm}

\noindent{\bf Case ${\mathcal A}$:} The convergence of $\Sigma_{n}$ to
$\Sigma_{\infty}$ has multiplicity one.

\vspace{.2cm}
\noindent{\bf Case ${\mathcal B}$:} The convergence of $\Sigma_{n}$ to
$\Sigma_{\infty}$ has multiplicity two.
\vspace{.2cm}

We  will consider both {\bf Case ${\mathcal A}$}  and {\bf Case ${\mathcal B}$}
simultaneously. However, our constructions in {\bf Case ${\mathcal B}$} will be
based on using only  the forming helicoids on the surfaces $M_n$ that actually pass
through the origin.  In a first reading of the following proof, we suggest that the
reader assume {\bf Case ${\mathcal A}$}  holds, as it is simpler
to follow the constructions  and the figures that we present
in this case.

 The Description~\ref{description} below follows from
 the smooth convergence of the $\S_n$ to $\S_\infty$ and because
 the statements in it hold for the limit vertical
helicoid $\S_\infty$ that contains the $x_3$-axis and $x_2$-axis.



\begin{ddescription}\label{description} {\em
Given $\ve_2\in(0,\frac12)$ and $N \in \N$, there exists $\overline{\omega}>0$
such that for any $\omega_1>\omega_2>\overline{\omega}\,$ there exist an
$n_0\in \N$ and positive numbers $r_n$, with $r_n=\frac{\sqrt2}{|A_{M_n}|(p_n)}$, such
that for any $n>n_0$ the following statements hold.
For clarity of exposition we abuse  notation and we let $M=M_n$ and $r=r_n$.
\ben[1.]\item $M\cap C(\omega_1 r,  \pi (N+1) r)$
consists of either one disk component, if
{\bf Case ${\mathcal A}$} holds, or two disk components, if
{\bf Case ${\mathcal B}$}
holds. One of the two possible disks in
$M\cap C(\omega_1 r,  \pi (N+1) r)$ contains the origin and we
denote it by $M(\omega_1 r)$. If {\bf Case ${\mathcal B}$}
holds, we denote the other component by $M^*(\omega_1 r)$.
\item $M(\omega_1 r)\cap C(\omega_2 r,  \pi (N+1) r)$ is
also a disk and we denote it by $M(\omega_2 r)$.
\item For any $t\in[-(N+1)\pi r,(N+1) \pi r]$, $M(\omega_1 r)$ intersects
the plane $\{x_3=t\}$ transversely in a single arc and when $t$ is
an integer multiple of $\pi r$,
this arc is disjoint from the solid vertical cylinder $C(\frac12, 1, \frac12 +\omega_2 r)$.
In particular, $M(\omega_1 r)\cap C(\frac12, 1, \frac12 +\omega_2 r)$
is a collection of $2N+2$ disks, each of
which is a graph over
\[
\{x_3=0\}\cap
C(\omega_1 r,1)\cap  C\left(\frac{1}{2},1,\frac{1}{2}+\omega_2 r\right).
\]
A similar description is valid  for $M^*(\omega_1 r)$, if {\bf Case ${\mathcal B}$} holds.
\item $M(\omega_1 r)\cap [C(\omega_1 r,  \pi (N+1) r)
-\Int (C(\omega_2 r,  \pi (N+1) r))]$, that is
\[
M(\omega_1 r)-\Int (M(\omega_2 r)),
\]
contains two oppositely oriented $N$-valued graphs $u_1$ and $u_2$ over
 $A(\omega_1 r, \omega_2 r)$. Moreover, these graphs $u_1$
 and $u_2$ can be chosen so that if {\bf Case ${\mathcal A}$} holds then
 the related regions between the sheets satisfy
\[
(\overline W[u_1]\cup\overline W[u_2])\cap M= {\rm graph}(u_1)\cup {\rm graph}(u_2),
\]
where $\overline W[u]$ denotes the closure of $W[u]$ in $\rth$.
In other words, no other part of $M$ comes between the sheets of
${\rm graph}(u_1)$ and ${\rm graph}(u_2)$.

If {\bf Case ${\mathcal B}$} holds,
$M^*(\omega_1 r)\cap [C(\omega_1 r,  \pi (N+1) r)-\Int (C(\omega_2 r,  \pi (N+1) r))]$
contains  another pair of oppositely oriented $N$-valued graphs $u^*_1$ and $u^*_2$ over
 $A(\omega_1 r, \omega_2 r)$ and
\[
(\overline W[u_1]\cup\overline W[u_2]\cup\overline W[u^*_1]\cup\overline W[u^*_2])\cap M=
\]
\[ {\rm graph}(u_1)\cup {\rm graph}(u_2)\cup {\rm graph}(u^*_1)\cup {\rm graph}(u^*_2).
\]

\item The separation between the sheets of the  $N$-valued
graphs $u_1$ and $u_2$ is bounded, i.e., for
$\rho_1,\rho_2\in [\omega_2r, \omega_1 r],\; |\theta_1-\theta_2|\leq 4\pi $ and $ i=1,2$,
\[
|u_i(\rho_1,\theta_1)-u_i(\rho_2,\theta_2)|<6\pi r.
\]

The same is true for $u^*_i$, if {\bf Case ${\mathcal B}$} holds.
\item $|\nabla u_i |<\ve_2$, $i=1,2$. The same is true for
$u^*_i$, if {\bf Case ${\mathcal B}$} holds.
\een }
\end{ddescription}

For the sake of completeness, in the discussion below we provide
some of the details that lead to the above description.

Let $\Sigma_\infty$ be the vertical helicoid containing the $x_2$ and $x_3$-axes
with  $|A_{\Sigma_\infty}|(\vec{0})=\sqrt 2$ and let $\Sigma_{n}$ be as defined in the
previous discussion. Then $\Sigma_{n}$ converges smoothly with multiplicity one or two
on compact subsets of $\rth$ to $\Sigma_\infty$. Hence for
any $\omega_1>\omega_2>0$ and $N \in \N$,  each of the
intersection sets $\Sigma_n\cap C(\omega_1 ,  \pi (N+1) )$ and
$\Sigma_n\cap C(\omega_2 ,  \pi (N+1) )$
consist of either one or two disk components that satisfy the
description in item~3, if $n$ is taken
sufficiently large. For simplicity, we provide some further details
when {\bf Case ${\mathcal A}$} holds. If $\overline{\omega}$ is
sufficiently large, given
$\omega_1>\overline{\omega}$, then $\Sigma_\infty\cap [C(\omega_1, \pi (N+1))
-\Int(C(\overline{\omega}, \pi (N+1)))]$
contains two oppositely oriented $N$-valued graphs $v_1$ and $v_2$ over
$A(\omega_1, \overline{\omega})$  such that
\[
|v_i(\rho_1,\theta_1)-v_i(\rho_2,\theta_2)|<5\pi, \, i=1,2,
\]
\[
\rho_1,\rho_2\in [\omega_2, \omega_1 ],\; |\theta_1-\theta_2|\leq 4\pi ,
\]
 and $|\nabla v_i |\leq \frac{\ve_2}{2}$, $i=1,2$ and
 nothing else is trapped between the sheets of
 ${\rm graph}(u_1)$ and ${\rm graph}(u_2)$ in the sense
 made precise by the previous description. Given
$\omega_2\in (\overline{\omega},\omega_1)$, because of the smooth
convergence, there exists $\overline{n}\in \N$ such that for any
 $n>\overline{n}$, then
 \[
\Sigma_n \cap [C(\omega_1,  \pi (N+1))-\Int(C(\overline{\omega},  \pi (N+1)))]
 \]
 contains  two oppositely oriented $N$-valued graphs $\wt{u}_1$ and $\wt{u}_2$ over
$A(\omega_1, \omega_2)$  such that
for any $k=1, \dots, N-1$,
\[
|\wt{u}_i(\rho_1,\theta_1)-\wt{u}_i(\rho_2,\theta_2)|<6\pi,
\]
\[
\rho_1,\rho_1\in [\omega_2, \omega_2 ],\; \theta_1,\theta_2\in[(-N +2k-2)\pi,  (-N +2k+2)\pi]
\]
and  $|\nabla \wt{u}_i |\leq \ve_2$, $i=1,2$. By definition
$\Sigma_n=\frac{1}{\sqrt 2} |A_{M_{n}}(\vec{0})|\widehat M_n$
and thus, rescaling proves that items 4, 5 and 6 hold.

Recall that a smooth domain in $\rth$
is {\em mean convex} if the mean curvature vector
of the boundary of the domain points into the domain.

\begin{definition} \label{defXM} {\em  $X_M$ is the closure of the
component of $\ov{\B}(\ve)-M$ such that
$\partial X_M$ is mean convex. }
\end{definition}
The
next lemma will be applied several times when
{\bf Case ${\mathcal B}$} holds and it gives some information on
the topology and geometry of $X_M\cap C(\omega_1 r,  \pi (N+1) r)$.

\begin{lemma} \label{caseB}  The set
$X_M\cap C(\omega_1 r,  \pi (N+1) r)$ has  one connected component
if  {\bf Case ${\mathcal A}$} holds, or it has
two connected components if {\bf Case ${\mathcal B}$}
holds.
\end{lemma}

\begin{proof}
If {\bf Case ${\mathcal A}$} holds, then $M\cap C(\omega_1 r,  \pi (N+1) r)$ consists of
a single disk and as $C(\omega_1 r,  \pi (N+1) r)$
is simply-connected, the lemma follows by elementary separation properties.

If  {\bf Case ${\mathcal B}$} holds, then
$M\cap C(\omega_1 r,  \pi (N+1) r)$ consists of two disjoint disks, meaning
that $\Sigma_{n} \cap C(\omega_1 ,  \pi (N+1))$ consists of two disjoint disks,
say $D_1$ and $D_2$, for $n$ large. Thus
$[C(\omega_1 ,  \pi (N+1))-\Sigma_{n}] \cap C(\omega_1 ,  \pi (N+1))$
consists of three connected components. Let $\Omega$ be the component of
$[C(\omega_1 ,  \pi (N+1))-\Sigma_{n}] \cap C(\omega_1 ,  \pi (N+1))$ such that
$\partial \Omega=D_1\cup D_2 \cup A$ where $A$ is the annulus
in $\partial C(\omega_1 r,  \pi (N+1) r)$
with boundary $\partial D_1\cup \partial D_2$.
In order to prove the lemma, it suffices to show that the normal vectors to $D_1$
and $D_2$ point towards the exterior of $\Omega$.  This property follows from the
earlier description that when the surfaces $\Sigma_n$ converge
with multiplicity two to the helicoid
$\Sigma_\infty$, then the region between them collapses and the
mean curvature vectors of the boundary surfaces
of this region point away from it; here the region $\Omega$
corresponds to a part of this collapsing region.
 This finishes the proof of the lemma.
\end{proof}

\begin{remark}\label{casebremark} {\em
Suppose {\bf Case ${\mathcal B}$} holds and let $X'_M$ be
one of the two connected components of
$X_M\cap C(\omega_1 r,  \pi (N+1) r)$. For later reference, we note that by elementary
separation properties, if $\gamma$ is an open arc
in $C(\omega_1 r,  \pi (N+1) r)$ with endpoints
$\partial \gamma \in \partial X'_M$ and
$\gamma\cap \partial X_M=\O$, then either
$\gamma\subset X'_M$ or $\gamma\subset [\ov{\B}(\ve)-X_M]$.}
\end{remark}

\subsection{Finding a minimal multi-valued graph on a fixed  scale}\label{findmvg}
In what follows we wish to use the highly sheeted multi-valued
graph forming on $M$ near the origin  to construct
a \underline{minimal} 10-valued  graph forming
near the origin that extends horizontally on a
scale proportional to $\ve$, and that lies on the mean convex side of $M$. Recall that $\ve<\frac12$.
The planar domain $M$ satisfies
Description~\ref{description} for certain constants
$\ve_2\in(0,\frac12)$, $\omega_1>\omega_2>\overline{\omega}>0$,
 $r>0$ and $N \in \N$ (for $n$ sufficiently large). These constants will be finally fixed toward
the end of this section. Nonetheless, in order for
some of the statements to be meaningful,
we will always assume $N$ to be greater than $m+4$ where $m$
is the number of boundary components of $M$.
Recall that $M(\omega_1 r)$ and $M(\omega_2 r)$ are
disks in $M$ and that each disk
resembles a piece of a scaled helicoid and contains the
origin, c.f., Description~\ref{description}.

Consider the intersection of
\[[{\rm graph}(u_1)\cup {\rm graph}(u_2)]\cap
C\left(\frac{1}{2},1,\frac{1}{2}+\omega_2 r\right);\]
 recall that $C(\frac{1}{2},1,\frac{1}{2}+\omega_2 r)$
 is the truncated solid vertical cylinder of radius $\frac 12$,
 centered at $(\frac 12+\omega_2 r, 0 ,0)$ with $|x_3|\leq 1$.
 This intersection consists of a collection of disk components \
\[
\Delta=\{\Delta_1, \ldots, \Delta_{2N}\},
\]
and each $\Delta_i$ is a graph over
\[
\{x_3=0\}\cap
C(\omega_1 r,1)\cap  C\left(\frac{1}{2},1,\frac{1}{2}+\omega_2 r\right).
\]

Because $M$ is embedded, the components of $\Delta$ can be assumed to
be ordered by their relative vertical heights and then,
by construction, the mean curvature vectors of consecutive components $\Delta_i$ and
$\Delta_{i+1}$    have oppositely signed $x_3$-coordinates.
Without loss of generality, we will henceforth assume that the surface
 $\partial C\left(\frac{1}{2},1,\frac{1}{2}+\omega_2 r\right)$ is
in general position with respect to $M$.

Let ${\mathcal
F}=\{F({1}),F({2}),\ldots, F({2N})\}$ be the ordered listing  of the components of $M\cap
C(\frac{1}{2},1,\frac{1}{2}+\omega_2 r)$ that intersect
the union of the disks in $\Delta$, and that are indexed so that $\Delta_i
\subset F(i)$ for each $i\in \{1,2,\ldots,2N\}$. Note that $\Delta_i$
and $\Delta_{i+j}$, for some $j\in \N$, may be contained in
the same component of $M\cap
 C(\frac{1}{2},1,\frac{1}{2}+\omega_2 r)$ and so,  $F(i)$ and $F(i+j)$
may represent the same set.

\begin{property}\label{property2}
 \ben[1.]\item Suppose $ i\in \{1,2,\ldots,2N-1\}$. If $F(i)\cap \partial M =\O$ and
the mean curvature vector of  $\Delta_i \subset F(i)$ is upward pointing, then
$ F({i+1})\cap
\partial M \not =\O$ or $F(i) =F({i+1})$.
\item Suppose $i\in\{2,3,\ldots, 2N\}$. If $F(i)\cap \partial M =\O$ and
the mean curvature vector of  $\Delta_i \subset F(i)$ is downward pointing, then
$ F({i-1})\cap
\partial M \not =\O$ or $F(i) =F({i-1})$. \een
\end{property}
\begin{proof}
We will prove the first property; the proof of the second property
is similar.

Suppose that $ i\in \{1,\ldots,2N-1\}$ and
the mean curvature vector of  $\Delta_i \subset F(i)$ is
upward pointing. Assume that $F({i+1})\cap \partial M=\O=F({i})\cap \partial M$,
and we will prove that $F(i) = F({i+1})$. Since  $F({i})\cap \partial M=\O$, then
$\partial F(i)\subset \partial C(\frac{1}{2},1,\frac{1}{2}+\omega_2 r)$
and so $F(i)$ separates the simply-connected domain $ C(\frac{1}{2},1,\frac{1}{2}+\omega_2 r)$
into two  connected domains.  Since the top and bottom disks in
$ \partial C(\frac{1}{2},1,\frac{1}{2}+\omega_2 r)$ are disjoint from $\ov{\B}(\ve)$
and lie in the same component of $C(\frac{1}{2},1,\frac{1}{2}+\omega_2 r)- F(i)$,
then the closure of one of these two  connected domains, which we denote by $X(F(i))\subset
C(\frac{1}{2},1,\frac{1}{2}+\omega_2 r)$, is disjoint from the
top and bottom disks of the solid cylinder and it follows that
$X(F(i))\subset \ov{\B}(\ve)\cap C(\frac{1}{2},1,\frac{1}{2}+\omega_2 r)$.

 For $t\geq 0$,
consider the family of surfaces
\[
\Omega_t=\partial C\left (\frac{1}{2},1,t+\frac{1}{2}+\omega_2
r\right).
\]
 The
maximum principle for $H=1$ surfaces applied to the family $\Omega_t$
shows that the last surface $\Omega_{t_0}$ which intersects $X(F(i))$,
intersects $F(i) \subset \partial X(F(i))$ at a point where the mean
curvature vector of $F(i)$ points into $X(F(i))$.  Hence, $F(i)$ is mean convex when
considered to be in the boundary of $X(F(i))$. 

Consider now a vertical line
segment $\sigma$ joining a point of the graph $\Delta_i\subset F(i)$ to a point
of the graph $\Delta_{i+1}\subset F({i+1})$; note that by Lemma~\ref{caseB}
and since the mean
curvature vector of $\Delta_i$ is upward pointing, the interior of $\sigma$ is
disjoint from $M$, independently of whether  {\bf Case~${\mathcal A}$} or
{\bf Case ${\mathcal B}$} holds. Since $X(F(i))$ is mean convex and the mean
curvature vector of $\Delta_i$ is upward pointing,
$\sigma$ is contained in $X(F(i))$. Since
$ F({i+1})\cap\partial M =\O$, then  the similarly defined
compact domain $X(F({i+1}))\subset C(\frac{1}{2},1,\frac{1}{2}+\omega_2
r)$ intersects  $X(F(i))$ at the point
$[\sigma \cap \Delta_{i+1}]\subset X(F({i+1}))$
and so, since $F(i+1)$ is either equal to or disjoint
from $F(i)$, then   $X(F({i+1}))\subset X(F(i))$.

Since $X(F_{i})$ intersects
$X(F({i+1}))$ at the point
$[\sigma \cap \Delta_i ]\subset X(F({i}))$,  the previous arguments imply
$X(F({i}))\subset X(F({i+1}))$. Because we have already shown
$X(F({i+1}))\subset X(F(i))$,  then $X(F(i))=X(F({i+1}))$, which
implies $F(i)=F({i+1})$. This completes the proof.\end{proof}

\begin{property}
\label{property3} There are at most $m-1$ indices $i$,  such that
$F(i) = F({i+1})$ and $F(i) \cap \partial M=\O$. \end{property}
\begin{proof}
Arguing by contradiction, suppose that there exist $m$ increasing
indices \linebreak $\{i(1), i(2), \ldots, i(m) \}$ such that for $j\in\{ 1, 2,\ldots,
m\}$,
\[
F(i(j))\cap \partial M= \O\text{ and } F(i(j))= F(i(j)+1).
\]
 Note that
for each $j\in\{ 1, 2,\ldots,
m\}$, $F(i(j))\cap M(\omega_1 r)$ contains the disks
$\Delta_{i(j)}$, $\Delta_{i(j)+1}$. Also note that since $F(i(j))\cap \partial M= \O$
 for $j\in\{ 1, 2,\ldots,
m\}$, then
$\cup_{j=1}^{m} \partial F(i(j))\subset \partial C\left(\frac12,1, \frac12 +\omega_2r \right)$

Let ${\mathcal F}=\{F(i(1)), F(i(2)), \ldots, F(i(m)) \}$ and let
$F_1, \ldots, F_{m'}$ be a listing of the distinct components in
${\mathcal F}$. For each $i=1,\dots,m'$, let $n_i\geq 2$ denote the number of components of
$F_i \cap M(\omega_1 r)$ and let $d_i $ denote the number of times
that $F_i$ appears in the list ${\mathcal F}$.  Note that
$n_i\geq d_i +1$. Recall that by item~3 of Description~\ref{description},
each component of $ M(\omega_1 r)\cap  C(\frac12, 1, \frac12 +\omega_2 r)$
is a disk which intersects
$\partial M(\omega_1 r)$ in a connected arc.

We next estimate the Euler characteristic of $M(\omega_1 r)\cup
\bigcup_{j=1}^{m}F(i(j))$ as follows. Using that
$\chi(F_i)\leq 1$ for each $i$ and $\sum_{i=1}^{m'} d_i =m$ gives
\[
\begin{split}
\chi \left(M(\omega_1 r)\cup\bigcup_{j=1}^{m}F(i(j))\right)
&=\chi(M(\omega_1 r))+  \chi\left(\bigcup_{i=1}^{m'}F_i\right)  -
\chi\left(M(\omega_1 r)\cap\bigcup_{i=1}^{m'}F_i\right)\\ & =1+
\sum_{i=1}^{m'} \chi(F_i) - \sum_{i=1}^{m'} n_i\\&\leq 1 + m'- \sum_{i=1}^{m'} (d_i+1) =1-m.
\end{split}
\]

Since the Euler characteristic of
a connected, compact orientable surface is $2-2g-k$ where $g$ is the genus and $k$ is
the number of boundary components, and since
$M(\omega_1 r)\cup\bigcup_{j=1}^{m}F(i(j))$ is a connected
planar domain, which means $g=0$, then the previous inequality
implies that the number of boundary components of   $M(\omega_1 r)\cup
\bigcup_{j=1}^{m}F(i(j))$ is at least $m+1$.
The hypothesis that
 for $j\in\{ 1, 2,\ldots,
m\}$,
$
F(i(j))\cap \partial M= \O$, implies that
each boundary component of $M(\omega_1 r)\cup
\bigcup_{j=1}^{m}F(i(j))$ is disjoint from
the boundary  of $M$.
Since $M$ is a planar domain, each of these boundary components separates $M$,
which  implies that $M-[M(\omega_1 r)\cup
\bigcup_{j=1}^{m}F(i(j))]$ contains at least $m+1$
components.
Since $M$ only has $m$ boundary components,
one of the components of  $M-[M(\omega_1 r)\cup
\bigcup_{j=1}^{m}F(i(j))]$,
say $T$, is disjoint from $\partial M$. Note that
 $\partial T \subset C(\frac{1}{2},1,\frac{1}{2}-\omega_1 r)$
 since $\partial T \subset [\partial M(\omega_1 r)
 \cup \bigcup_{j=1}^{m} \partial F(i(j))]\subset [\ov{\B}(\ve)\cap  C(\frac{1}{2},1,\frac{1}{2}-\omega_2 r)]$.

Suppose for the moment that $\partial T\cap \partial M(\omega_1 r)\neq \O$,
and we will arrive at a contradiction.
In this case, $\partial T $ contains at least one point $p$ in
$\partial M(\omega_1 r) \cap \partial C(\frac{1}{2},1,\frac{1}{2}-\omega_1 r)$
and since $T$ is disjoint from $M(\omega_1 r)$, $T$ contains points outside
$C(\frac{1}{2},1,\frac{1}{2}-\omega_1 r)$  near $p$.
For $t\geq0$, consider the family of translated surfaces
\[
\Omega_t=\partial
C\left(\frac{1}{2},1, \frac{1}{2}- \omega_1 r \right) + ( -t,0,0).
\]
 Since the last such
translated surface $\Omega_{t_0}$ which intersects $T$, does so at a
point in the interior of $T$ and $T$ is contained on the mean convex
side of $\Omega_{t_0}$ near this point, a standard application of the
maximum principle for $H=1$ surfaces gives a contradiction. This
contradiction proves that $\partial T\cap \partial M(\omega_1 r)= \O$.

Since  we may now assume that $\partial T\cap \partial M(\omega_1 r)= \O$, then
\[
\partial T \subset \bigcup_{j=1}^{m} \partial F(i(j))=\bigcup_{j=1}^{m} \partial F(i(j))\cap \ov{\B}(\ve)
\subset C\left(\frac{1}{2},1,\frac{1}{2}+\omega_2 r\right).
\]

Let $p\in \partial T$.
Since for some $j$, $p\in \partial F(i(j)) \subset \partial C(\frac12,1,\frac12 +\omega_2 r)$
and also since $T$ is disjoint from $F(i(j))\subset C(\frac12,1,\frac12+\omega_2 r)$,
then $T$ contains points outside of $C(\frac{1}{2},1,\frac{1}{2}+\omega_2
r)$.
A straightforward modification of the arguments in the previous
paragraph using the maximum principle applied to the family of translated surfaces
\[
\Omega_t=\partial
C\left(\frac{1}{2},1, \frac{1}{2}+ \omega_2 r\right)  + ( -t,0,0)
\] gives a contradiction. This
contradiction completes the proof that Property~\ref{property3} holds. \end{proof}

The next two propositions imply that if the $N$-valued
graphs in $M$ forming nearby the origin contains a sufficiently large number
of sheets, then it is possible to find a disk in $M$
whose boundary satisfies certain geometric properties.

\begin{property}
\label{property4} Suppose that there exist $m+1$ indexed domains $$F(i(1)),\, F(i(2)),
\, \ldots , F(i(m+1)),$$ not necessarily
distinct as sets, with increasing indices and each of which
intersects a fixed boundary component $\nu$ of $\partial M$ and such that the
mean curvature vectors of the subdomains $\Delta(i(1)),\, \Delta(i(2)),
\, \ldots , \Delta(i(m+1))$ are all upward pointing or all downward pointing. Then
there exists a collection $\Gamma = \{ \g(1), \g(2), \ldots,
\g(m+1) \}$ of $m+1$  pairwise disjoint,
embedded  arcs $\g(j)\subset F(i(j))$ with end points in $\partial F(i(j))$
such that: \ben[1.]\item  For each $j$,
$\g(j)\cap M(\omega_1 r)=\g(j)\cap C(\omega_1 r,  \pi (N+1) r) $
is the arc $\Delta(i(j))\cap \{x_2=0\}$ and $\g(j)$ has the point $p_j$ in
$\Delta(i(j))\cap \partial C(\omega_2 r, 1)$ as one of its end points. The boundary
of $\g(j)$ consists of $p_j$ and a point in
 $\nu$.
\item $M- (M(\omega_2 r)\cup \bigcup _{j=1}^{m+1} \g(j))$
contains a  component whose closure  is a disk
$D$ with  $\partial D$ consisting of an arc  $\a\subset \nu$,
two arcs in $\Gamma$ and an arc $\be$ on  the component $\tau$ of
$\partial M(\omega_2 r)\cap \partial C(\omega_2r,1)$ that intersects $\Delta(i(1))$.
\item Furthermore,
if for each $j\in \{1, 2, \ldots, m\}$, $i(j+1)-i(j)\geq 2\wt{N}+2$ for some $\wt{N}\in \N$, then
$D\cap(M(\omega_1 r)-\Int(M(\omega_2 r)))$ contains an $\wt{N}$-valued graph over the
annulus $A(\omega_1 r,\omega_2r)$.
\een
\end{property}

\begin{proof}
For each $j\in \{1, 2, \ldots, m+1\}$, consider an embedded arc
$\g(j)$ in $F(i(j))$ joining the point $p_j  \in \Delta(i(j))\cap
\partial C(\omega_2 r, \pi (N+1)r)$ to a  point in $\nu $;
this is possible since $F(i(j))$ intersects $\nu$.
Since $F(i(j))- M(\omega_1 r)$ is path connected,
then one can choose $\g(j)$ to intersect
$C(\omega_1 r,  \pi (N+1) r)$ in the arc $\Delta(i(j))\cap \{x_2=0\}$.
As
the arcs in $\Gamma =\{\g(1), \g(2), \ldots , \g(m+1) \}$ can also be
constructed to be pairwise disjoint, it is
straightforward to check that item~1 holds for the collection
$\G$.

Since  the mean curvature vectors of the domains
$\Delta(i(1)),\,  \ldots , \Delta(i(m+1))$  are all upward
or all downward pointing, without loss of generality we can assume that
  all of the points $p_j$ lie on
  \[
\tau:=  \partial {\rm graph}(u_1)\cap \partial C_{\omega_2 r},
\]
that is the inner boundary of ${\rm graph}(u_1)$.

Since $M$ is a planar domain,
$M- (M(\omega_2 r)\cup \Gamma)$ contains $m+1$
components. Because $M$ has $m$ boundary components, it follows that
at least two
of these $m+1$ components are disjoint from $\partial M -
\nu$ and the closure of one of these components, say $D$,
intersects $\tau$ in a subarc $\be$ and it does not intersect
$ \partial M(\omega_2 r)-\beta$. It  follows that $D$ is a disk with
boundary as described in
item 2 of Property~\ref{property4}.

Item~3 follows immediately from the  construction of $D$. \end{proof}

\begin{proposition} \label{disk}
Suppose that $2N\geq m(2m+2)(m + 2\widetilde{N} +2)$. Then there exists a compact disk $D\subset M$
such that $D\cap(M(\omega_1 r)-\Int(M(\omega_2 r)))$ contains an $\widetilde{N}$-valued graph
$u^+$ over the annulus $A(\omega_1 r,\omega_2 r)$. Moreover, there exist two indices $i$ and $j$,
with $i-j \geq 2\widetilde{N}+2$, such that  $\partial D$ consists of four arcs
$\a, \sigma_1, \sigma_2, \beta$ satisfying:
\ben[1.]\item  $\a = \partial D \cap \partial M\subset \partial\B(\ve)$;
 \item $\sigma_1=\g(i)$ is an arc in $ F(i)$ with end points in $\partial F(i)$, and such that \\
 $\g(i)\cap M(\omega_1 r)=\g(i)\cap C(\omega_1 r,  \pi (N+1) r)  =\Delta(i)\cap \{x_2=0\}$;
\item $\sigma_2=\g(j)$ is an arc in $ F(j)$ with end points in $\partial F(j)$, and such that \\
 $\g(j)\cap M(\omega_1 r)=\g(j)\cap C(\omega_1 r,  \pi (N+1) r)  =\Delta(j)\cap \{x_2=0\}$;
\item $\beta$ is an arc in
$\partial M(\omega_2 r)\cap \partial C(\omega_2 r,1)$.
\een
\end{proposition}

\begin{proof}
Since
$\cF=\{F(1), \ldots,F(m(2m+2)(m + 2\widetilde{N} +2)), \ldots, F(2N)\}$, then
for $l\in \{1,2,\ldots, m(2m+2) \}$, the family of
domains
\[
T_l = \{ F(i) \mid i \in \{ (l-1)(m +2\widetilde{N} +2)+1, \ldots,
(l-1)(m +2\widetilde{N} +2) +m \}\}
\]
is a well-defined subset of $\cF$, each $T_l$ consists
of $m$, not necessarily distinct, indexed elements,
and if $F(i)\in T_{l+1}$ and $F(j)\in T_l$, then $i-j\geq 2\wt{N}+2$.
Properties~\ref{property2} and
~\ref{property3} imply that there exists an element $F(f(l))\in T_l$
such that $F(f(l)) \cap \partial M \neq \O$. Thus, the collection
$$\{ F(f(1)), F(f(2)), \ldots, F(f(m(2m+2)))\}$$ has $m(2m+2)$ indexed
elements  and for each $l\in\{1,2,\ldots, m(2m+2)-1 \}$, $f(l+1)-f(l)\geq 2\wt{N}+2$.

Since $M$ has $m$ boundary components, then there exists an ordered
subcollection $$\{F(i(1)), F(i(2)), \ldots, F(i(2m+2))\}, $$ with
each element in this ordered subcollection intersecting some particular
component $\nu$ of $\partial M$. Therefore, there exists a further ordered subcollection
$$\{F(k(1)), F(k(2)), \ldots, F(k(m+1))\} $$ for which the mean curvature vectors
of the disks $\Delta(k(1)),\, \Delta(k(2)),
\, \ldots , \Delta(k(m+1))$ are all upward pointing or all downward pointing.
 By construction, for each
$j\in \{1, 2, \ldots, m+1\}$, $k(j+1)-k(j)\geq 2\widetilde{N}+2$.

Proposition~\ref{disk} now follows from
Property~\ref{property4} applied to the collection
$$\{F(k(1)), F(k(2)), \ldots, F(k(m+1))\}. $$
\par \vspace{-.2cm} \end{proof}

Recall that $X_M$ is the closure of the
connected mean convex region of
$\ov{\B}(\ve)-M$. The next lemma is an immediate consequence of the main theorem in~\cite{my2}.
It says that because the domain $X_M$ is mean convex,
it is possible to find a stable embedded minimal disk in $X_M$ with the
same boundary as $D$, where $D$ denotes the disk given in
Proposition~\ref{disk}.

\begin{lemma} \label{sdisk}  Let $D$ denote the disk given in
Proposition~\ref{disk}. Then there is a stable
minimal disk $E$ embedded  in $X_M$ with $\partial E =\partial D$.
\end{lemma}

In what follows we   show that $E$ contains a highly-sheeted multi-valued graph
on a small scale near the origin and that some of the sheets
of this multi-valued graph extend horizontally on a scale proportional to
$\ve$ when $\wt{N}$ is sufficiently large, where $\wt{N}$ is
described in Proposition~\ref{disk}. To do
this we   need to apply two results of Colding and Minicozzi.
The first result is a  gradient estimate for
certain stable minimal surfaces in thin slabs, which follows  from an application of the
curvature estimates by Schoen~\cite{sc3} for stable orientable
minimal surfaces.

\begin{lemma}[Lemma I.0.9. in~\cite{cm21}]\label{gradest}
Let $\Gamma\subset\{|x_3|\leq \beta h\}$ be a compact stable embedded
minimal surface and let $T_h(\Pi(\partial \Gamma))\subset \R^2$
denote the regular $h$-neighborhood of the projection $\Pi(\partial
\Gamma)$ of $\partial
\Gamma$ to $\R^2$, where $\Pi\colon \G \to \R^2$ is orthogonal
projection to the $(x_1,x_2)$-plane. There exist $C_g,\beta_s>0$ so that if
$\beta\leq\beta_s$ and $F$ is a component of
$$\mathbb{R}^2 - T_h(\Pi(\partial \Gamma)),$$ then each component of
$\Pi^{-1}(F)\cap\Gamma$ is a graph over $F$ of a function $u$ with
$$|\nabla_{\R^2}u|\leq C_g\beta. $$
\end{lemma}

The second result needed is a scaled version of Theorem II.0.21
in~\cite{cm21} that gives conditions under which an embedded
stable minimal disk contains a large multi-valued graph.

\begin{theorem}[Theorem II.0.21 in~\cite{cm21}]\label{2vgminimal}
Given $\tau>0$, there exist $N_1,\Omega_1>1$ and $\ve_1>0$ such that the following holds.

Given $\delta \in (0,1)$, let $\Sigma\subset \B(R_0)$ be a stable embedded
minimal disk with $\partial \Sigma \subset \B(\de r_0)\cup \partial \B({R_0})\cup \{x_1=0\}$
where $\partial \Sigma - \partial \B(R_0)$ is connected.
Suppose the following hold:
\ben[1.]\item $\Omega_1 r_0<1< \frac{R_0}{ \delta\Omega_1}$.
\item $\Sigma$ contains an $N_1$-valued graph $\Sigma_g$
over $A(\de,\de r_0)$ with norm of the gradient is less than or equal to $ \ve_1$.
\item $\Pi^{-1}(\{(x_1,x_2,0)\mid x_1^2 +x_2^2 \leq (\de r_0)^2\})
\cap \Sigma_g^M\subset \{|x_3|\leq \ve_1 \de r_0\};$
here $\Sigma_g^M$ denotes the middle sheet of $\Sigma_g$.
\item An arc $\widetilde \eta$ connects
$\Sigma_g$ to $\partial \Sigma \backslash \partial \B(R_0)$,
where $\widetilde\eta\subset \Pi^{-1}(D_{\delta r_0})\cap [\Sigma - \partial \B(R_0)] $.
\een
Then $\Sigma$ contains a 10-valued graph $\Sigma_d$ over
$A( R_0\slash\Omega_1, \de r_0)$ with norm of the gradient less than or equal to $ \tau$.
\end{theorem}

\begin{remark}{\em
This version of Theorem II.0.21 in~\cite{cm21} is
obtained by scaling $\S_g$ by the factor $\delta$.
While in the statement of
Theorem~II.0.21 in~\cite{cm21}, $\Sigma_g$ is said to contain a 2-valued graph, the result
above where $\Sigma_g$ is said to contain a 10-valued graph also holds.
}\end{remark}

In reading the next
statement, recall that $M$ and $r$ are elements of a sequence
that depends on $n$ and that for
convenience we have omitted the index $n$. Among other
things the next theorem states that the minimal disk $E$ in
Lemma~\ref{sdisk} contains a 10-valued
graph on a fixed horizontal scale (when $n$ is sufficiently large).

\begin{theorem}\label{extmvg}
Given $\tau\in(0,\frac12)$ there exists an $\Omega_1=\Omega_1(\tau)>1$,
$\omega_2=\omega_2(\tau)$ and
$ \omega_1=  \omega_1(\tau)>10\omega_2$ such that
the following is true of $M$ when $n$ is sufficiently large.
\ben[1.]\item There exists a minimal subdisk $E'\subset E$
containing a 10-valued  graph $E'_g$ over $A(\frac{\ve}{\Omega_1},4\omega_2r)$
with norm of the gradient less than $\tau$.
\item The intersection  $C_{\omega_1r} \cap M\cap W[E'_g]$
consists of exactly two or four 9-valued graphs with norms of the gradients less
than $\frac{\tau}{2}$. (See Definition~\ref{def:W} for the definition of $W[E'_g]$.)
\item There exist  constants $\beta_1,\beta_2\in (0,1]$ such that if
$\tau<\frac{\beta_2}{1000}$, then for any  $p=(x_1,x_2,0)$
such that $x_1^2+x_2^2=(\frac{\ve}{10\Omega_1(\tau)})^2$
the following holds. The intersection set
\[
C_{\beta_1\alpha}(p)\cap M\cap W[E'_g]
\]
with $\alpha=\frac{80\ve\tau}{\Omega_1(\tau)\beta_2}$  is
non-empty and it consists of at least eight disconnected components.
\een
\end{theorem}
\begin{proof}
Before beginning the actual proof of Theorem~\ref{extmvg},
we summarize some of the results that we have obtained
so far. For simplicity, we suppose that
 Case ${\mathcal A}$ holds.

Recall that  in this case $M_n$ separates $\ov{\B}(\ve)$
into two components and $X_{M_n}$ denotes the closure of the
component with mean convex boundary.
Given $\ve_2\in(0,\frac12)$ and $\wt N \in \N$,
there exist $N\in \N$, $\overline{\omega}>0$
such that for $\omega_1>5\omega_2>\overline{\omega}$ there exist an
$n_0\in \N$ and positive numbers $r_n$, with
$r_n= \frac{\sqrt 2}{|A_{M_n}|(\vec 0)}\to 0$ as $n\to \infty$, such
that for any $n>n_0$ the following statements hold;
see Description~\ref{description} for more details.
Again, for clarity of exposition we abuse  the
notations and we let $M=M_n$ and $r=r_n$.
\ben[1.]\item $M\cap C(\omega_1 r,  \pi (N+1) r)$
consists of the disk component $M(\omega_1 r)$ passing through the origin.
\item $M(\omega_1 r)\cap C(\omega_2 r,  \pi (N+1) r)$ is also a disk.
\item $M(\omega_1 r)\cap [C(\omega_1 r,
\pi (N+1) r)-\Int (C(\omega_2 r,  \pi (N+1) r))]$, that is
\[
M(\omega_1 r)-\Int (M(\omega_2 r)),
\]
contains two oppositely oriented $N$-valued graphs $u_1$ and $u_2$ over
 $A(\omega_1 r, \omega_2 r)$ and
\[
[\overline W[u_1]\cup\overline W[u_2]]\cap M
= {\rm graph}(u_1)\cup {\rm graph}(u_2).
\]

\item The separation between the sheets of the  $N$-valued
graphs $u_1$ and $u_2$ is bounded,  c.f., for
$\rho_1,\rho_2\in [\omega_2r, \omega_1 r],\; |\theta_1-\theta_2|\leq 4\pi$ and $i=1,2$,
\[
|u_i(\rho_1,\theta_1)-u_i(\rho_2,\theta_2)|<6\pi r.
\]
\item $|\nabla u_i|<\ve_2$, $i=1,2$.
\een

Moreover, by Proposition~\ref{disk} and Lemma~\ref{sdisk},
the $N$-valued graph $u_1$ contains an $\wt N$-valued
subgraph graph$(u^+)$ and there exists an embedded stable
minimal disk $E$ disjoint from $M$ whose boundary
$\xi_o\cup\xi_i\cup\sigma_1\cup\sigma_2$ satisfies
the following properties; see  Figure~\ref{summary}:
\bit
\item $\xi_o\subset \partial M\subset \partial \B(\ve)$, the ``outer'' boundary;
\item  $\xi_i\subset \partial M(\omega_2 r)\cap \partial  C_{\omega_2 r}$, the
``inner'' boundary, is an arc in the inner boundary of ${\rm graph}(u^+)$, see
Definition~\ref{multigraph};
\item $\sigma_1\cup\sigma_2\subset \{x_1\geq \omega_2 r \}$, the ``side'' boundaries;
\item  $M(\omega_1r)\cap\sigma_1=\{x_1>0,x_2=0\}\cap u^+[\wt{N}]$
and $M(\omega_1r)\cap\sigma_2=\{x_1>0,x_2=0\}\cap u^+[1]$.
\eit

Recall  that $u^+[k]$ denotes
the $k$-th sheet of the $\wt{N}$-valued graph $u^+$; see
Definition~\ref{def:middle}.
Recall also that $N$ can be taken to be $m(m+1)(m + 2\widetilde{N} +2)$,
where $m$ is the number of boundary components of $M$; see
Proposition~\ref{disk}. Without loss of generality, we may
assume that $\wt{N}$ is odd and the normal vector to $u^+$
is downward pointing and that $u_2$ contains
an $\wt{N}$-valued subgraph $u^-$ satisfying the following properties:
\ben[1.] \item
$[\overline W[u^+]\cup\overline W[u^-]]\cap M= {\rm graph}(u^+)\cup {\rm graph}(u^-).$
\item $u^+_{\rm mid}>u^-_{\rm mid}$.
\een
Furthermore after a small vertical translation of $
M$ by $(2\pi r)y$, for some $y\in [-2(N+2), \ldots 2(N+2)]$, we will assume
$u^+_{\rm mid}$ intersects $\{x_3=0\}$.

\begin{figure}
\begin{center}
\includegraphics[width=11cm]{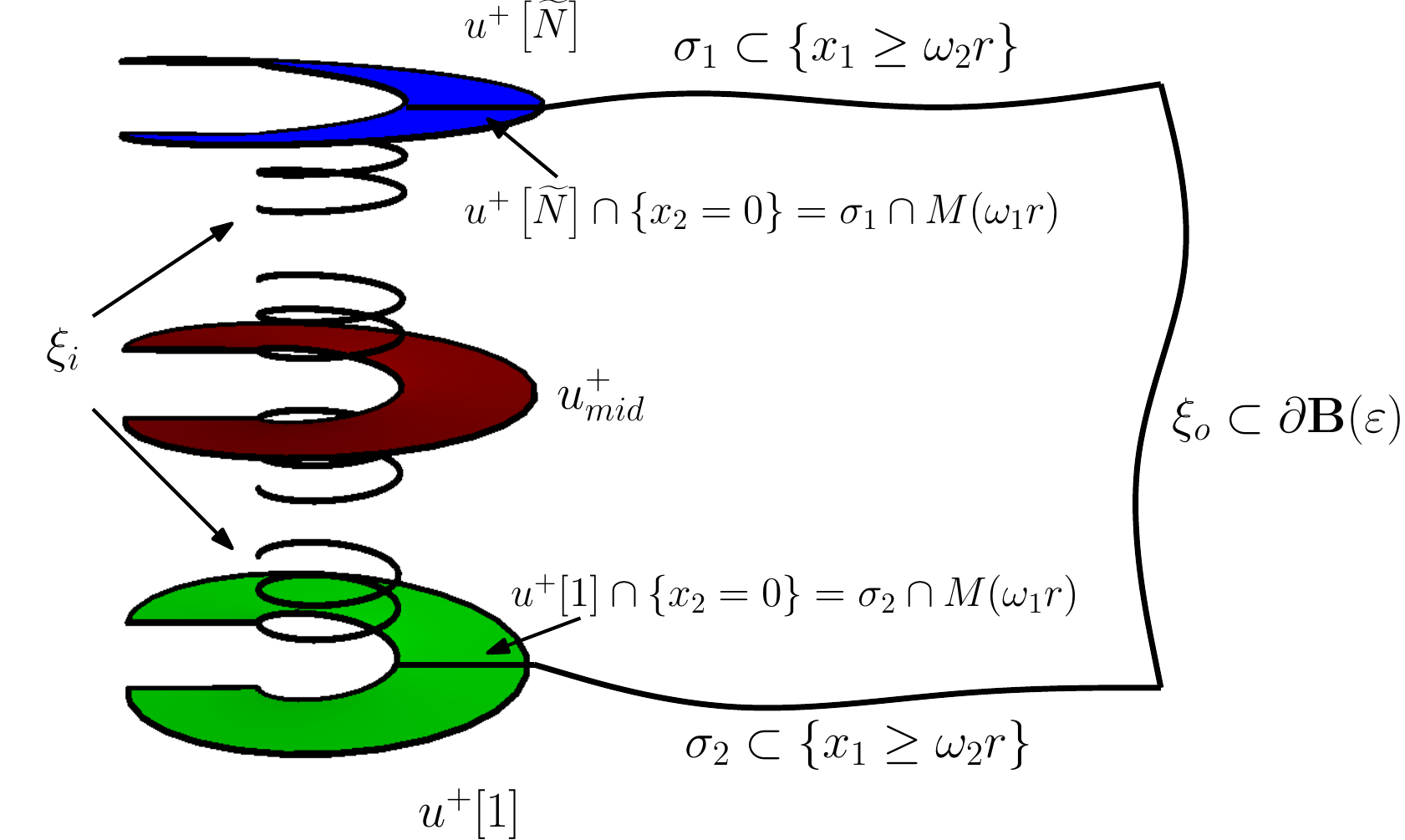}
\caption{}
\label{summary}
\end{center}
\end{figure}

We now begin the actual proof. Let $N_1=N_1(\tau)$, $\ve_1=\ve_1(\tau)$ and
$\Omega_1=\Omega_1(\tau)$ be as given by Theorem~\ref{2vgminimal}. The idea of
the proof is to find $\wt N,\omega_1,\omega_2, \ve_2,r$ in the previous description
such that a subdisk $E'\subset E$ satisfies the hypotheses of
Theorem~\ref{2vgminimal}, and thus extends horizontally
on a fixed scale proportional to $\ve$.

  For the next discussion, refer to Figure~\ref{rho}.
\begin{figure}
\begin{center}
\includegraphics[width=7.5cm]{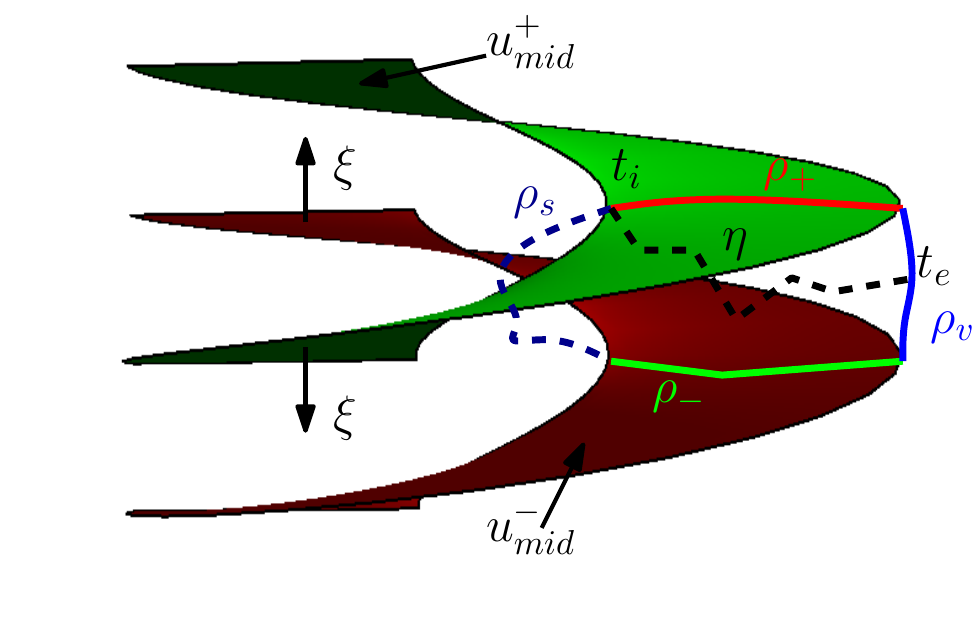}
\caption{}
\label{rho}
\end{center}
\end{figure}
Fix $\wt N=N_1+5$ and consider the simple closed curve
$\rho=\rho_v\cup \rho_+\cup\rho_s\cup\rho_-$ where
  \[
  \rho_\pm\colon [\omega_2r, 3\omega_2r] \to {\rm graph}(u^\pm_{\rm mid}),
  \]
such that for $t\in [\omega_2r, 3\omega_2r]$,
\[
\rho_\pm(t)=\{(t,0,x_3)\mid  x_3\in (-\infty, \infty)\}\cap {\rm graph}(u^\pm_{\rm mid}).
\]
The arc $\rho_v$ is the open vertical line segment
connecting $\rho_+(3\omega_2r)$ and $\rho_-(3\omega_2r)$.
The arc $\rho_s$ is an arc in $M(\omega_2r)$ connecting
$\rho_+(\omega_2r)$ and $\rho_-(\omega_2r)$;
by item~3 in Description~\ref{description}, $\rho_s$ can be chosen
to be contained in the slab
$\{ x_3(\rho_{-} (\omega_2 r))\leq x_3\leq x_3(\rho_{+} (\omega_2 r))\}$
and can be parameterized by its $x_3$-coordinate.
Note that
by Description~\ref{description} and Lemma~\ref{caseB},
$\rho_v\subset  X_M$,  $ \rho_+\cup\rho_s\cup\rho_-\subset \partial X_M$
and  $\rho$ is the boundary of a disk in $X_M\cap C_{3\omega_2 r}$.
(This is also true if {\bf Case ${\mathcal B}$} holds.)
Recall that $C_{3\omega_2 r}$
is the vertical solid cylinder centered at the origin of radius $3\omega_2 r$;
see the definition at the beginning of this section.
The main result
in~\cite{my2} implies that $\rho$ is the boundary of an embedded least-area
disk $D(\rho)$ in $X_M\cap C_{3\omega_2 r}$, which we
may assume is transverse to the disk $E$.  In particular, since by construction
$\rho$ intersects $\partial E$ transversely
in the single point $t_i=u^+(\omega_2r)\in\xi_i$,
and $ \rho_+\cup\rho_s\cup\rho_-\subset \partial X_M$, there
exists an arc $\eta\subset D(\rho)\cap E $ with one end
point $t_i$, the ``interior'' point,  and its other end point
$t_e\in \rho_v\cap E$, the ``exterior'' point.

Consider  the connected
component $\Gamma_{t_e}$ of $C_{2\omega_2r}(t_e)\cap E$
that contains $t_e$; recall that $C_{2\omega_2r}(t_e)$
is the solid vertical cylinder centered at $t_e$ of radius $2\omega_2 r$. Then, $\Gamma_{t_e}$
is contained in a slab of height less then $6\pi r$.
Using Lemma~\ref{gradest}, it follows that $E$ is locally graphical at $t_e$
with norm of the gradient less than $ \frac{6\pi C_g}{\omega_2}$ (take
$\beta h=6\pi r$ and $h=\omega_2 r$). After prolongating this graph, following
the multi-valued graph ${\rm graph}(u^+)$, we find that $E$ contains an $(N_1+2)$-valued
graph $E_g$ over $A(\omega_1 r-\omega_2 r,3\omega_2 r)$ satisfying:
 $$\mbox{the norm of the gradient of $E_g$ is less than } \frac{6\pi C_g}{\omega_2}.$$

The intersection set
\[
\{(0,x_2, x_3)\mid x_2>0, x_3\in(-\infty,\infty)\}\cap E_g
\]
consists of $(N_1+2)$ arcs $\gamma_i$, $i=1, \dots, N_1+2$,
where the order of the indexes agrees with the relative
heights of the arcs. Let $\gamma_+$ denote $\gamma_{N_1+2}$
and let $\gamma_-$ denote $\gamma_{1}$. Let $p_\pm^e$
denote the endpoint of $\gamma_\pm$ in
$\partial E_g\cap\partial C_{\omega_1 r-\omega_2 r}$ and let $p_\pm^i$ denote
the endpoint of $\gamma_\pm$ in
$\partial E_g\cap \partial C_{3\omega_2 r}$.
Without loss of generality, we will assume
that the plane $\{x_1=0\}$ intersects $E$ transversally.
Let $\Gamma_+$ denote the connected arc in $[E-\Int(\gamma_+)]\cap \{x_1=0\}$
containing $p^e_+$ and let $\Gamma_-$ denote the
connected arc in $[E-\Int(\gamma_-)]\cap \{x_1=0\}$ containing $p^e_-$. Since
$\Gamma_+\cup \g_+$ and $\Gamma_-\cup\g_-$ are
planar curves in the minimal disk $E$, neither of these curves can be closed in $E$
by the maximum principle. This implies that
$\partial \Gamma_+=\{p_+^e, q_+\}$ and $q_+\in \partial E$,
$\partial \Gamma_-=\{p_-^e, q_-\}$ and $q_-\in \partial E$.
Since the boundary of $E$ is a subset of
\[
\partial C_{\omega_2 r}\cup \{x_1>\omega_2 r\}\cup \partial\B(\ve)
\] and $\Gamma_\pm\subset \{x_1=0\}$, it follows that
\[
q_\pm\in\partial E \cap [\partial C_{\omega_2 r}\cup \partial\B(\ve)]=\xi_i\cup\xi_o.
\]

\begin{claim}
The interiors of the arcs $\G_\pm$ are disjoint
from $E_g\cup \eta$. Moreover, $q_\pm\in\xi_o$.
\end{claim}
\begin{proof}
We will prove that  $\G_+-\{p_+\}$ is disjoint from
$E_g\cup\eta\cup\xi_i$, from which the claim follows for $\G_+$.
The proof of the case for $\G_-$ is similar and will be left to the reader.

Arguing by contradiction, suppose
$\G_+-\{p_+\}$ is not disjoint from $E_g\cup\eta\cup\xi_i$. Assuming that
$\Gamma_+$ is parameterized beginning at $p_+^e$,
let $r_+\in \Gamma_+ \cap [E_g\cup \eta\cup\xi_i]$
be the first point along $\G_+$ in this intersection set and
let $\Gamma_+'$ be the closed  arc of $\Gamma_+-\{r_+\}$
such that $\partial \Gamma_+'=\{r_+, p_+^e\}$.

\begin{figure}[h]
\begin{center}
\includegraphics[width=11cm]{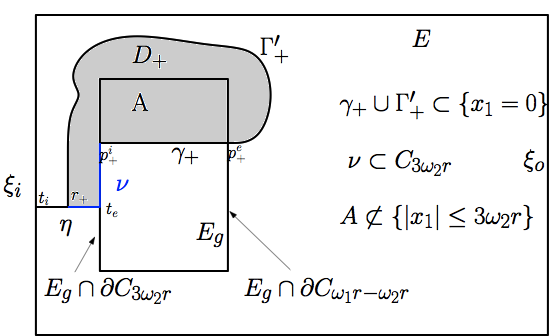}
\caption{}
\label{nu1}
\end{center}
\end{figure}

First suppose $r_+\in \eta$, which is the case described
in Figure~\ref{nu1} . Then there exists an embedded arc
\[
\nu\subset \eta\cup [E_g\cap \partial C_{3\omega_2r}]
\]
 connecting $r_+$ and $p_+^i$. Recall that $E_g\cap \partial C_{3\omega_2r}$ is
 the inner boundary of the multi-valued graph $E_g$. Therefore,
\[
\Gamma=\Gamma_+'\cup \nu \cup \gamma_+
\] is a simple closed curve that bounds a disk $D_+$ in $E$.
By the convex hull property, $D_+$ is contained in the
convex hull of its boundary. In particular, $D_+$ is
contained in the vertical slab $\{|x_1|\leq 3\omega_2 r \}$. Note however that
$\Gamma_+$ separates the $(N_1+2)$-th sheet of $E_g$,
that is the set $u^+[N+2]$, into two components that are not contained in the
slab $\{|x_1|\leq3\omega_2 r \}$ because $\omega_1 r>4\omega_2 r$.
By elementary separation properties, one of these components must be
in $D_+\subset \{|x_1|\leq3\omega_2 r \}$, which gives a contradiction.

An argument similar to that presented in the previous paragraph shows
$[\Gamma_+-\{p_+^e\} ] \cap E_g\cap \partial C_{3\omega_2r}=\O$ by using a similarly defined
embedded arc $\nu\subset [E_g\cap \partial C_{3\omega_2r}]$ connecting $r_+$ to $p^i_+$.

Next consider the case that
$r_+\in E_g\cap \partial C_{\omega_1 r-\omega_2r}$.  In this case,
$r_+\in \partial C_{\omega_1 r-\omega_2r}$ is the
end point of some component arc $\g$ of $\{x_1=0\}\cap E_g$ with the
other end point of $\g$ being $p_\g \in \partial E_g\cap \partial C_{3\omega_2 r}$.
Letting $\nu\subset \partial E_g\cap \partial C_{3\omega_2 r} $
be the arc with end points $p_\g$   and $p^i_+$, we find that
\[
\Gamma=\Gamma_+'\cup \g_j \cup \nu \cup \gamma_+
\] is a simple closed curve that bounds a disk
$D_+$ in $E$.  Similar arguments as in the previous two paragraphs
then provide the desired contradiction that the
minimal disk $D_+$ cannot be contained in the convex hull of its boundary.

Finally, suppose that $r_+\in \xi_i$.  In this case,
we find the desired simple closed curve
\[
\Gamma=\Gamma_+'\cup \nu \cup \gamma_+,
\]
where

\[
\nu\subset \xi_i\cup \eta\cup [E_g\cap \partial C_{3\omega_2r}]
\]
is an embedded arc that connects
the points $r_+$ with  $p^i_+$, see Figure~\ref{nu2}.
\begin{figure}[h]
\begin{center}
\includegraphics[width=12cm]{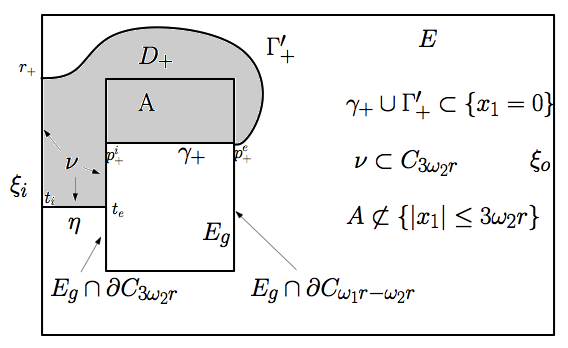}
\caption{}
\label{nu2}
\end{center}
\end{figure}
Arguing as previously,
we obtain a contradiction and the claim is proved.
\end{proof}

Denote by $\nu$  the arc in
$\partial E_g\cap\partial C_{3\omega_2r}$ connecting $p_+^i$
and $p_-^i$ and let $\alpha$ be
the arc in $\xi_o$ connecting $q_+$ to $q_-$. The simple closed curve
\[
\Gamma_+\cup\gamma_+\cup\nu\cup\gamma_-\cup\Gamma_-\cup\alpha
\]
bounds a subdisk $E'$ in $E$ which contains the
$N_1$-valued graph $E_g'=\bigcup_{k=2}^{2n-1}u^+[k]\subset E_g$ over
$A(\omega_1r-\omega_2r,3\omega_2r)$; see Figure~\ref{eg}.
\begin{figure}[h]
\begin{center}
\includegraphics[width=12cm]{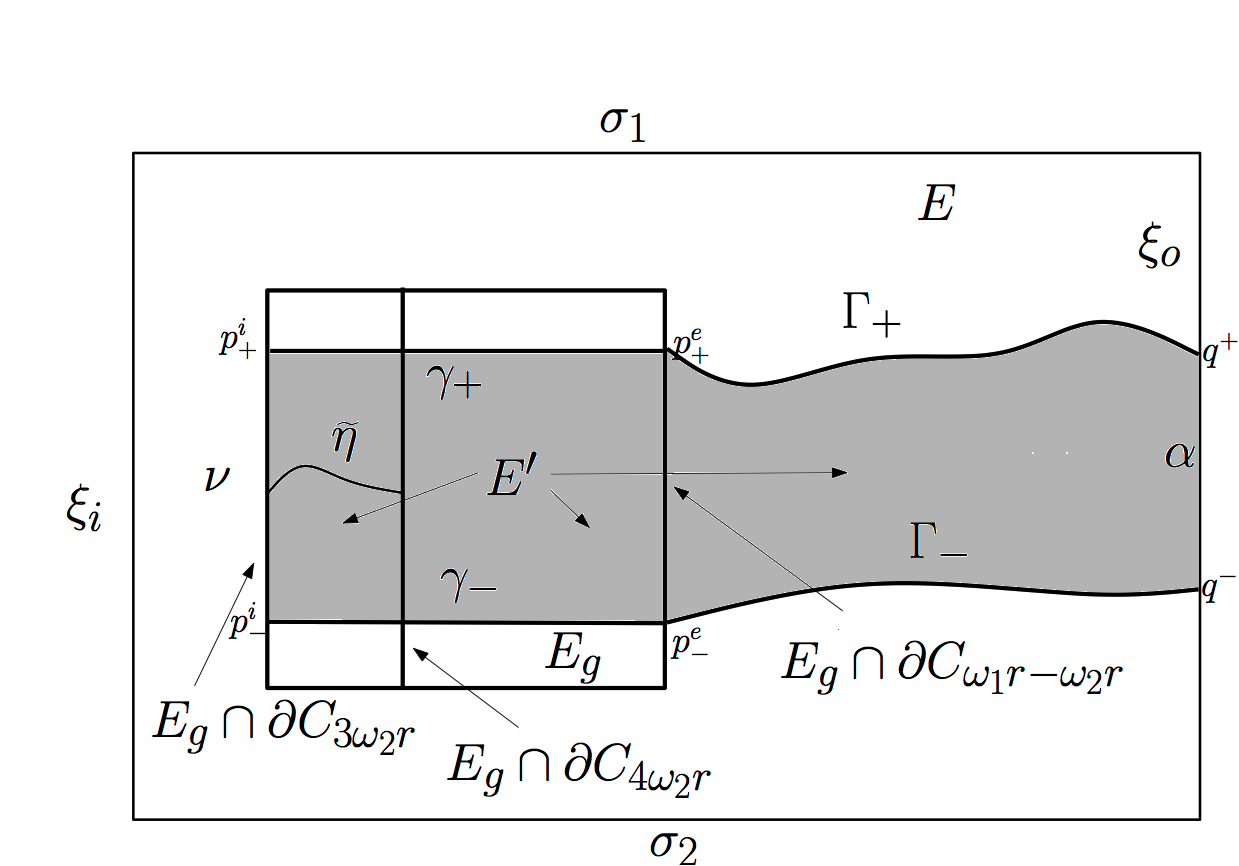}
\caption{}
\label{eg}
\end{center}
\end{figure}
The norm of the gradient of this $N_1$-valued graph $E_g'$ is also less than
$\frac{6\pi C_g}{\omega_2}$. Note that if $\omega_2 > \bar{\omega}_2>4\pi (N+2)$,
where $\bar{\omega}_2$ is chosen sufficiently large, then
\[
\partial E'\subset \B(4\omega_2 r)\cup\{x_1=0\}\cup\partial \B(\ve).
\]
The disk $E'$ is the stable minimal disk to which we will apply Theorem~\ref{2vgminimal}.

By construction $E'$
contains the $N_1$-valued subgraph $\widetilde{E}'_g$ of $E'_g$
over the annulus $A(\omega_1 r-\omega_2r, 4\omega_2 r)$
with norm of the gradient less than
$\frac{6\pi C_g}{\omega_2}$. Thus, if $\omega_2>\frac{6\pi C_g}{\ve_1}$, $E'$
satisfies item~2 of Theorem~\ref{2vgminimal}
 with
\[
\delta=(\omega_1-\omega_2)r,\, \, \delta r_0
=4\omega_2 r \,\text{ and } \,r_0=\frac{4\omega_2}{\omega_1-\omega_2}.
\]

Clearly, as described in Figure~\ref{eg}, there exists a curve
$\widetilde \eta\subset \Pi^{-1}(\{(x_1,x_2,0)\mid x_1^2 +x_2^2 \leq (4\omega_2 r)^2\})$
connecting $\widetilde{E}'_g$ to
$\partial E' - \partial \B(\ve)$ and thus $E'$
also satisfies item~4 of Theorem~\ref{2vgminimal}.

Recall that $r$ depends on $n$ and
it goes to zero as $n\to\infty$.
After choosing $\omega_1$ sufficiently large, and then
choosing  $n$ much larger, we obtain the inequalities
\[
\frac{4\omega_2}{\omega_1-\omega_2}\Omega_1<1
\text{ and }
\frac{\ve}{r(\omega_1-\omega_2)}\geq1,
\]
 that together imply that the minimal disk $E'$ satisfies item~1 of Theorem~\ref{2vgminimal}.

Finally, let $\wt u_{\rm mid}$ denote the middle sheet of
$\widetilde{E}'_g$. Using the
bound $\frac{6\pi C_g}{\omega_2}$ for the norm of the gradient gives that
\[
\Pi^{-1}(\{(x_1,x_2,0)\mid x_1^2 +x_2^2 \leq (4\omega_2 r)^2\})\cap \wt{u}_{\rm mid} \subset
\left \{|x_3|\leq 20\pi \omega_2 r \cdot \frac{6\pi C_g}{\omega_2}\right\}
\]
and if $\omega_2>\frac{30\pi^2 C_g}{\ve_1} $
then $20\pi \omega_2 r \cdot \frac{6\pi C_g}{\omega_2}<\ve_14\omega_2 r$; in other words,
 $E'$ satisfies
item~3 of Theorem~\ref{2vgminimal}.

Recall that $M$ and $r$ depend on $n$ and that $ \ov{\omega}>4\pi (N+2)$ is
 chosen sufficiently large so that
 $$ \partial E'\subset \B(4\omega_2 r)\cup\{x_1=0\}\cup\partial \B(\ve).$$
 We now summarize what we have shown so far.
If  $n$ is sufficiently large, then the quantities discussed at the
beginning of the proof can be assumed to satisfy
the following conditions:

\[
 \wt N= N_1+5,\,\; \omega_2 > \max \left \{  \bar{\omega}_2, \frac{30\pi^2 C_g}{\ve_1}\right\},
\]
\[
\omega_1 > 4\omega_2\Omega_1+10\omega_2 \,\text{ and } \,
r <  \frac{\ve}{\omega_1-\omega_2},
\]
and then the stable minimal disk $E'$ satisfies the hypotheses of
Theorem~\ref{2vgminimal} and thus it contains a 10-valued graph $E'_g$
over $A(\ve\slash\Omega_1,4\omega_2r)$ with norm of the
gradient less than $\tau$. This completes the proof of item~1 in the theorem.

The proof that the intersection
$C_{\omega_1r} \cap M\cap W[E'_g]$ consists of exactly two or four 9-valued graphs
follows   from the construction of $W[E'_g]$ and
Description~\ref{description}; note that because
of embeddedness and by construction, the multi-valued graphs on $M$ near the origin and in
$C_{\omega_1r} $ spiral together with
the minimal multivalued graph $E_g'$. Whether there are two or four
9-valued graphs depends on whether the convergence
to the helicoid detailed in Description~\ref{description} is with
multiplicity one or two. The norms of the gradients
of such graphs is bounded by $\frac{\tau}{2}$ as long
as $\ve_2$ is. This completes
the proof of item~2 in the theorem.

The proof of item~3 will use the existence of the
minimal 10-valued graph $E_g'$ and a standard dragging argument. First note that
because of the gradient estimates for the minimal
10-valued graphs, a simple calculation shows
that $W[E_g']$ is contained in the open cone
\[
\mathcal C=\left\{(x_1,x_2,x_3)\mid |x_3|< 80\tau \sqrt{x_1^2+x_2^2}\right\}.
\]

We begin the proof of item~3 by proving the existence of certain embedded
domains of vertical nodoids, where nodoids are the nonembedded
surfaces of revolution of nonzero constant mean curvature defined by Delaunay~\cite{de1}
and where vertical means that  the $x_3$-axis is its axis of revolution .

\begin{lemma} \label{lem6.1}
There exists constants $\beta_1,\beta_2\in (0,1]$ such that the following
holds. For $s\in (0,1]$ consider the circles
$C_1=\{(x_1,x_2,x_3)\mid x_1^2+x_2^2=\beta_1^2 s^2, \, x_3
=-\beta_2 s\}$ and $C_2=\{(x_1,x_2,x_3)\mid x_1^2+x_2^2=\beta_1^2 s^2, \, x_3=\beta_2 s \}$.
Then there exists a subdomain $\cN_s$ of a vertical nodoid
with constant mean curvature $\frac1s$ such that $\cN_s$ is embedded
with boundary $C_1\cup C_2$, and $\cN_s$  is
contained in the convex hull of its boundary.
\end{lemma}
\begin{proof}[Proof of Lemma~\ref{lem6.1}]
The lemma follows by rescaling, after finding the correct
numbers $\beta_1, \beta_2$ for $s=1$ and
 a compact embedded portion $\cN_1$ of a vertical
nodoid with constant mean curvature
one and
such that  $\cN_1$  is contained in the boundary of its convex hull.
\end{proof}

Note that $\cN_s$ separates $C_{\beta_1 s}$ into a bounded and
an unbounded component and  its mean curvature vector points
into the bounded component.
Given $p\in\rth$ we let $\cN_s(p)$ be $\cN_s$ translated
by $p$. Suppose now that $p=(x_1,x_2,0)$ such that
$x_1^2+x_2^2=(\frac{\ve}{10\Omega_1(\tau)})^2$. Note that
if $\tau<\frac{\beta_2}{1000}$, then $\cN_\alpha(p)$ with
$\alpha=\frac{80\ve\tau}{\Omega_1(\tau)\beta_2}$ satisfies
\[
\partial \cN_\alpha(p)\cap \mathcal C=\O
\]
and its mean curvature, $\frac 1\alpha$, is greater than one.

Consider the point
$p=(\frac{\ve}{10\Omega_1(\tau)},0,0)$. Let
\[
\Gamma=\left\{(x_1,x_2,0)\,:\, \left |\frac{x_2}{x_1}\right | <2, \, x_1>0,\,
(4\omega_2 r)^2<x_1^2+x_2^2<\left (\frac{\ve}{\Omega_1(\tau)}\right )^2\right\}.
\] Note that $E_g'$ separates $\Gamma\times \mathbb R$
into nine bounded components and that
\[
 \cN_\alpha(p)\subset \Gamma \times \mathbb R.
 \]
Let $\Delta\subset \mathcal C$ be one of these bounded components.

In order to prove item~3 of the theorem, we will first show that
\[
\Delta\cap M\cap C_{\beta_1\alpha}(p)\neq\O.
\]
Since $ \cN_{\alpha}(p)\subset  C_{\beta_1\alpha}(p)$, it
suffices to show that $\Delta\cap M\cap \cN_{\alpha}(p)\neq\O$.
Consider the family of rescaled nodoids
$\cN^t=\cN_{t\alpha}(tp)$ for $t\in (0,1]$. Since $\partial \cN^t\cap \Delta=\O$
and the mean curvature of $\cN^t$ is greater than one, by using a so-called dragging argument,
it suffices to show that   there exists some $\ov{t}$  small so that
$\Delta\cap M\cap \cN^{\ov{t}}$ contains an interior point
of $M\cap \Delta$ and $\cN^{\ov{t}}\subset \G\times \R$. This is because by an application of the
mean curvature comparison principle, the family of
nodoids $\cN^t$, $t\in [\ov{t},1]$, cannot have  a last point of interior contact with
$M\cap \Delta$. Recall that $C_{\omega_1 r}\cap\Delta \cap M$
contains at least one component which is a graph  over
$
\left\{(x_1,x_2,0)\mid \left |\frac{x_2}{x_1}\right
| <2,\, x_1>0, \,(4\omega_2 r)^2<x_1^2+x_2^2<(\omega_1 r )^2\right\}
$ and therefore a simple calculation shows that by
taking $\overline t=\frac{10\Omega_1(\tau)}{\ve}\omega_1r$, then
\[
\Delta\cap M\cap \cN^{\overline t}
\]
contains an interior point of $M\cap \Delta$. Furthermore, by our earlier
choice of  $\omega_1 > 4\omega_2\Omega_1+10\omega_2$,
we conclude  that $\cN^{\ov{t}}\subset \G\times \R$.

Using the fact that $\Delta\cap M\cap \cN_{\alpha}(p)\neq\O$,
a straightforward further prolongation argument,
by moving the center $p$ of $\cN_\a(p)$
along the circle centered at the origin of radius $|p|$,
finishes the proof of item~3, which completes the proof of Theorem~\ref{extmvg}.
 \end{proof}

\subsection{Extending the constant mean curvature
multi-valued graph to a scale proportional
to $\ve$ }
\label{sec3.3}

In this section we reintroduce the subscripted indexes for the sequence of surfaces $M_n$.
We show that for $n$ sufficiently large,
$M_n$ contains two, oppositely oriented 3-valued graphs on a fixed horizontal scale
and with the norm of the gradient small.
This is an improvement to the
description in Section~\ref{sec3} where the multi-valued graphs formed on the
scale of the norm of the second fundamental form;
the next theorem was inspired by and generalizes Theorem II.0.21 in~\cite{cm21}
to the nonzero constant mean curvature setting.

\begin{theorem}\label{extmvg2}
Given $\tau_2>0$, there exists $\Omega_2=\Omega_2(\tau_2)$
and $\omega_2=\omega_2(\tau_2)$ such that for  $n$ sufficiently large,
the surface
$M_n$ contains  two  oriented 3-valued graphs $G_n^{up},G_n^{down}$ over
$A(\ve\slash\Omega_2,4\omega_2\frac{\sqrt2}{|A_{M_n}|(\vec 0)})$
with norm of the gradient less than $\tau_2$,
where $G_n^{up}$ is oriented by an upward pointing normal and
$G_n^{down}$ is oriented by a downward pointing normal. Furthermore,
these  3-valued graphs can be chosen to lie between the
sheets of the 10-valued minimal graph $E'_g(n)$
given in Theorem~\ref{extmvg} and  so that
$G_n^{up}\cap W[G_n^{down}]$ is a 2-valued graph.
\end{theorem}

 \begin{proof}
Recall that after normalizing the surfaces $M_n$ by
rigid motions that are expressed as translations  by vectors of length
at most $\ov{\ve}$ for any particular small choice $\ov{\ve}\in (0,\ve/4)$,
composed with rotations
fixing the origin,
we  assume that the surfaces $M_n$ satisfy $|A_{M_n}|(\vec 0)>n$
and that the origin is a point of almost-maximal curvature on $M_n$
around which one or two vertical helicoids are forming  in $M_n$ on the scale of
$|A_{M_n}|(\vec 0)$.  It is in this situation that we apply  Theorem~\ref{extmvg}
to obtain the 10-valued minimal graph $E'_g(n)$ described in the
statement of Theorem~\ref{extmvg2}.

By  Theorem~\ref{extmvg}, for each $l\in\N$, there exist
 \[
\ov{n}(l)>2,\,\, \Omega_1(l)>1, \,\, \omega_2(l)>0,\,\,
 \omega_1(l)>10\omega_2(l) 
 \]
 such that for $n>\ov{n}(l)$, $M_n$ contains a minimal 10-valued graph $E'_g(n,l)$ over
 $$A(\frac{\ve}{\Omega_1(l)},4\omega_2(l)\frac{\sqrt2}{|A_{M_n}|(\vec 0)})$$
 with the norms of the gradients bounded by $\frac1l$; we will also assume
 for all $l\in \N$ that
 $\ov{n}(l+1)>\ov{n}(l)\in \N$ and that, after replacing by a subsequence and reindexing,
the inequality
$$4\omega_2(l)\frac{\sqrt2}{|A_{M_n}|(\vec 0)}< \frac{\ve}{n\Omega_1(l)}$$ also holds  when $n>\ov{n}(l)$;
in particular, under this assumption the ratios of the  outer radius to the inner radius of
the  annulus over which the 10-valued minimal multigraph $E'_g(n,l)$ is defined
go to infinity as $n\to \infty$, and
$\lim_{n\to \infty} 4\omega_2(l)\frac{\sqrt2}{|A_{M_n}|(\vec 0)}=0$.
 Furthermore, by item~2 of Theorem~\ref{extmvg}, for  $n>\ov{n}(l)$, the intersection
$$
 C_{\omega_1(l)\frac{\sqrt2}{|A_{M_n}|(\vec 0)}} \cap M_n\cap W[E'_g(n,l)]
 $$
 consists of exactly two or four 9-valued graphs with the norms of
 the gradients bounded by $\frac1{2l}$.

Let $\tau>0$ be given and let
 \[
 W_{n,l} := \{p\in \ov{W}[E'_g(n,l)]\mid -3\pi\leq \theta\leq 3\pi\}.
\]
Then, for $n>\ov{n}(l)$,
$C_{\omega_1(l)\frac{\sqrt2}{|A_{M_n}|(\vec 0)}}\cap M_{n}\cap  W_{n,l} $
consists of a collection $\cC_{n,l}$  of
either two or four
3-valued graphs with the norms of the gradients bounded from above
by $\frac{1}{2l}$.

We claim that  for some $l_\tau \in \N$ sufficiently large,
and given $n> \ov{n}(l_\tau)$, then
the 3-valued graphs in $\cC_{n,l_\tau}$ extend horizontally to
3-valued graphs over $$A(\frac{\ve}{n(l_\tau)\Omega_1(l_{\tau})},4\omega_2(l_{\tau})\frac{\sqrt2}{|A_{M_n}|(\vec 0)})$$
with the norms of the gradients less than $\tau$.
This being the case, define $G_n^{up}$, $G_n^{down}$ to be the two related extended graphs
that intersect $M_n(\omega_1)$ and which have their normal vectors
pointing up or down, respectively. Then, with respect these choices,
the remaining statements of the theorem can be easily verified
to hold.

Hence, arguing by contradiction suppose that the claim
fails for some $\tau>0$.
Then for every $l\in \N$ sufficiently large, there exists a surface
$M_{n(l)}$ with $n(l)>\ov{n}(l)$
such that the following statement holds: For $$r_{n(l)}=\frac{\sqrt2}{|A_{M_{n(l)}}|(\vec 0)},$$
the 3-valued graphs in $\cC_{n(l),l}$ do not extend horizontally
as 3-valued graphs over the annulus
$$A(\frac{\ve}{n(l)\Omega_1(l)},4\omega_2(l)r_{n(l)})$$
with the norms of the gradients less than $\tau$,
where by our previous choices,
$$4\omega_2(l)r_{n(l)}<\frac{\ve}{n(l)\Omega_1(l)}<\frac{\ve}{2\Omega_1(l)}.$$

Thus for any fixed $l$ large enough  so that $M_{n(l)}$ exists, let
$$\rho({n(l)})\in [4\omega_2(l)r_{n(l)},\frac{\ve}{2\Omega_1(l)}] $$
be the supremum of the set of
numbers $\rho\in [4\omega_2(l)r_{n(l)},\frac{\ve}{2\Omega_1(l)}] $
such that for any point
$p\in  C_{\rho}\cap  M_{n(l)}\cap W_{n(l),l}$, the tangent plane to $M_{n(l)}$ at $p$
makes an angle less than  $\tan ^{-1}(\tau)$  with the $(x_1,x_2)$-plane.
Note that $\rho(n(l))\geq\omega_1(l) r_{n(l)}$ because of
the aforementioned properties of the surfaces in $\cC_{n(l),l}$.

 Let
\[
\Pi_{n(l)}\colon  W_{n(l),l} \to [4\omega_2(l)r_{n(l)},\ve/\Omega_1(l)]\times [-3\pi,3\pi]
\]
denote the natural projection. The map $\Pi_{n(l)}$ restricted to
$ \Int({C}_{\rho({n(l)})})\cap M_{n(l)}\cap W_{n(l),l}$ is a proper submersion
and thus the preimage of a sufficiently small neighborhood of a point
\[
(\rho,\theta)\in [4\omega_2(l)r_{n(l)},\rho({n(l)}))\times [-3\pi, 3\pi]
\]
 consists of exactly two or four graphs.
Let $p_{n(l)}=(\rho(n(l)),\theta_{n(l)},x_3^{n(l)})\in
  W_{n(l),l}\cap M_{n(l)}$ be a point where the tangent plane
of $M_{n(l)}$ makes an angle greater than or equal to
$\tan^{-1}(\tau)$ with the $(x_1,x_2)$-plane and let $T_{n(l)}$
be the connected component of
\[
M_{n(l)}\cap \B\left(p_{n(l)},\frac{\rho(n(l))}{2}\right)
\]
containing $p_{n(l)}$. Because of the gradient
estimates for the 10-valued minimal graph $E'_g(n(l),l)$, $T_{n(l)}$
is contained in a horizontal slab of height at
most $18\rho(n(l))\frac1l$. Furthermore, we remark that
$\partial T_{n(l)}\subset \partial \B(p_{k(l)},\frac{\rho(n(l))}{2})$
and $\Pi_{n(l)}$ restricted to
\[
\Int(C_{\rho(n(l))})\cap T_{n(l)}\cap   W_{n(l),l},
\]
is at most  four-to-one.

For each $n(l)$, consider the rescaled sequence
$\widetilde{T}_{n(l)}=\frac{1}{\rho_{n(l)}}T_{n(l)}$. We claim that the
number of boundary components of $T_{n(l)}$, and
thus of $ \widetilde{T}_{n(l)}$, is bounded from
above by the number of boundary components
of $M_{n(l)}$ which is at most $m$. Otherwise, since $M_{n(l)}$
is a planar domain, there exists a component $\Lambda$ of
$M_{n(l)}- T_{n(l)}$ that is disjoint from
$\partial M_{n(l)}$ and contains points outside the
ball $\ov{\B}(p_{n(l)},\frac{\rho(n(l))}{2})$.
Since the mean curvature $H\leq 1$ and $\ve<\frac 12$, an application of
the mean curvature comparison principle with
spheres centered at $p_{n(l)}$
implies that $\Lambda$ contains points outside
of   $\ov{\B}(\ve)$. However $\Lambda \subset M_{n(l)} \subset \ov{\B}(\ve)$ and this
contradiction proves the claim.

Note that
$\partial \widetilde{T}_{n(l)}\subset \partial \B(\frac{1}{\rho(n(l))}p_{n(l)},\frac 12 )$,
and that the
constant mean curvatures of the surfaces
$\widetilde{T}_{n(l)}$ are going to zero as $l\to \infty$. We next
apply some of the previous results contained in this
paper, e.g. Theorem~\ref{extmvg}, to study the geometry of the planar domain
$\widetilde T_{n(l)}$ near the point $\frac{1}{\rho(n(l))}p_{n(l)}$,
which by our choices lies in $ \partial C_1$.

The surfaces
$\widetilde{T}_{n(l)}$ are contained in  \underline{horizontal} slabs
of height at most $\frac{18}{l}$. Moreover, there exist
rigid motions $\mathcal R_{n(l)}$ that are each a
  translation composed with a rotation around the $x_3$-axis,
such that the following  hold:\ben
 \item $\mathcal R_{n(l)}(\frac{1}{\rho(n(l))} {p}_{n(l)})=\vec 0$.

 \item For
$\displaystyle \Gamma=\left\{(x_1,x_2,0)\mid x_1>0,\,\,
x_1^2+x_2^2< \frac 14,\,\, x_2 <\frac {x_1} 4\right \}
$, \\
$\displaystyle
(\Gamma\times\R) \cap \mathcal R_{n(l)}(\widetilde{T}_{n(l)})
$
 consists of at least two and at most four components, each of which is graphical
 over $\Gamma$. Note that in addition to other properties, $\G$ and $\mathcal R_{n(l)}$
 are chosen so that
 $\Pi_{n(l)}[\rho(n(l))\mathcal R_{n(l)}^{-1}(\G)] \subset [ C_{\rho(n(l))}\cap \{x_3=0\}]$. 
\een
 Let $\overline T_{n(l)}=\mathcal R_{n(l)}(\widetilde{T}_{n(l)})$
 and   note that $\partial \overline{T}_{n(l)}\subset \partial \B( \frac 12 )$.

Since the height of the slab containing $\overline{T}_{n(l)}$
is going to zero as $l\to\infty$ and
the tangent plane at $\vec 0$ makes an angle of
at least $\tan^{-1}(\tau)$ with the $(x_1,x_2)$-plane,   it follows that as
$l$ goes to infinity, the norm of the
second fundamental form of $\overline{T}_{n(l)}$
is becoming arbitrarily
large at certain points in $\overline{T}_{n(l)}$
converging to $\vec 0$.
Using this property
that as $l$ goes to infinity, the norm of the second fundamental
form of $\overline{T}_{n(l)}$ is becoming
arbitrarily large nearby $\vec 0$, we will  prove that
$\overline{T}_{n(l)}$ must
intersect the region $\Gamma\times\R$
in more than four components, which will produce the desired  contradiction.

After replacing by a subsequence and normalizing the surfaces by
translating by vectors  $\vec{v}_{n(l)}$, $\vec{v}_{n(l)} \to \vec{0}$, Theorem~\ref{extmvg}
implies that there exist
a fixed rotation $\mathcal R$ and
constants $\beta_1,\beta_2\in (0,1]$ such that  the following holds:
Given $\tau_1<\frac{\beta_2}{1000}$ there exists $\lambda_{\tau_1}\in \N$, $\omega_2(\tau_1)$
and  $\Omega_1(\tau_1)$ such that  for $l>\lambda_{\tau_1}$:
\begin{enumerate}
\item There exists a 10-valued minimal graph $E'_g(l, \tau_1)$ over
$$A(\frac{1}{2\Omega_1(\tau_1)},4\omega_2(\tau_1)\frac{\sqrt2}{|A_{\cR(\ov{T}_{n(l)})}|(\vec 0)},)$$
with norm of the gradient less than $\tau_1$ (item 1 of Theorem~\ref{extmvg});
\item With
$\alpha=40\frac{1}{\Omega_1(\tau_1)}\frac{\tau_1}{\beta_2 }$
and $p=(x_1,x_2,0)$ with $x_1^2+x_2^2= ( \frac{1}{20\Omega_1(\tau_1)}  )^2$,
the intersection
\[
C_{\beta_1\alpha}(p)\cap \mathcal {R}(\overline{T}_{n(l)})
\cap \left \{(x_1,x_2,x_3)\mid |x_3|\leq \frac{40 \tau_1}{\Omega_1(\tau_1)}\right \},
\]
 is non-empty and contains at least eight disconnected
components (item 3 of Theorem~\ref{extmvg}).
\end{enumerate}

We claim that $\mathcal R$ is a rotation around the $x_3$-axis.
Arguing by contradiction, suppose that this is not the case.
Note that for any $\tau_1<\frac{\beta_2}{1000}$,
as $l$ goes to infinity, the slab containing
$\mathcal {R}(\overline{T}_{n(l)})$  converges to a
plane $\mathcal P$ through the origin that is not the $(x_1,x_2)$-plane
as the slab was horizontal before applying $\cR$.
Therefore $E'_g(l, \tau_1)$ converges to a disk in $\mathcal P$.
Let $\theta$ denote the angle that $\mathcal P$ makes with the
$(x_1,x_2)$-plane and pick
$$\tau_1=\min(\tan (\theta/2), \frac{\beta_2}{1000}).$$ This choice of
$\tau_1$ leads to a contradiction because $E'_g(l, \tau_1)$ cannot
converge as a set to a disk in $\mathcal P$ and have norms of the gradients
bounded by $\tau_1$. This contradiction proves that the rotation $\mathcal R$
is a rotation around the $x_3$-axis.

Finally, one obtains a contradiction by finding
$p=(x_1,x_2,0)$ with $x_1^2+x_2^2= ( \frac{1}{20\Omega_1(\tau_1)}  )^2$
and taking $\tau_1$ sufficiently small such that the
disk centered at $p$ of radius $\beta_1\alpha$ is contained in $\mathcal R(\Gamma)$. This
leads to a contradiction because on the one hand,
$C_{\beta_1\alpha}(p)\cap\mathcal R(\overline{T}_{n(l)})$ consists
of at least eight
components. On the other hand $\mathcal{R}(\Gamma\times\R) \cap\mathcal{R}(\overline{T}_{n(l)})$
consists of at most four components, each of which is graphical. This last
contradiction completes the proof of Theorem~\ref{extmvg2}.
 \end{proof}

Theorem~\ref{extmvg}, Theorem~\ref{extmvg2} and
their proofs not only demonstrate that there exists a 3-valued graph in $M_n$
that extends on a fixed, horizontal scale when
$n$ is large and they also give the following, detailed geometric
description of a planar domain with constant mean curvature, zero flux and
large norm of the second fundamental form at
the origin. In the following theorem,
we summarize what we have proven so far.
The precise meaning of certain statements are made
clear from the previous results and proofs.

\begin{theorem}\label{mainextension}
Given $\ve,\tau>0$, $\ov{\ve}\in (0,\ve/4)$ and $m\in \mathbb N$, there
exist  constants $\Omega_\tau=\Omega(\tau,m)$,
$\omega_\tau=\omega(\tau,m)$ and  $G_\tau=G(\ov{\ve}, \ve,\tau,m) $
such that if $M$ is a  connected compact $H$-planar domain with zero
flux, $H\in (0,\frac 1{2\ve})$, $M\subset \ov{\B}(\ve)$,
$\partial M\subset \partial \B(\ve)$ and consists of at most $m$ components, $\vec 0\in M$ and
$|A_M|(\vec 0)>\frac{1}{\eta}G_\tau$, for $\eta\in (0,1]$,
then for any  $p\in \ov{\B}(\eta \ov{\ve})$ that is a maximum of the
function
$|A_{M}|(\cdot)(\eta\bar\ve-|\cdot|)$, after translating
$M$ by $-p$,
the following geometric description of $M$ holds.
\begin{itemize}

\item On the scale of the norm of the second
fundamental form  $M$ looks like one or two helicoid nearby  the
origin and, after a rotation that turns these helicoids into
vertical helicoids, $M$  contains a 3-valued graph
$u$ over  $A(\ve\slash\Omega_\tau,\frac{\omega_\tau}{|A_M|(\vec 0)})$
with norm of the gradient less than $\tau$.

\item The intersection $W[u]\cap [M -{\rm graph}(u)]$
contains an oppositely oriented 2-valued
graph $\widetilde u$ with norm of the gradient less
than $\tau$ and $\B(10\frac{\omega_\tau}{|A_M|(\vec 0)})\cap M$  includes a
 disk $D$  containing the interior boundaries of ${\rm graph}( u)$ and ${\rm graph}( \widetilde u)$.

\item If near the origin $M$ looks like one  helicoid,
then the previous description is accurate, namely
$W[u]\cap [M -{\rm graph}(u)]$ \underline{consists} of an oppositely oriented 2-valued
graph $\widetilde u$ with norm of the gradient less
than $\tau$ and $\B(10\frac{\omega_\tau}{|A_M|(\vec 0)})\cap M$  \underline{consists} of a
 disk $D$  containing the interior boundaries of ${\rm graph}( u)$ and ${\rm graph}( \widetilde u)$.

\item If near the origin $M$ looks like two helicoids, then
$W[u]\cap [M -{\rm graph}(u)\cup {\rm graph} (\widetilde u)]$ consists of
a pair of oppositely oriented 2-valued graphs $u_1$ and
$\widetilde u_1$ with norm of the gradient less than $\tau$ and
$\B(10\frac{\omega_\tau}{|A_M|(\vec 0)})\cap[ M-D]$
consists of a disk containing the inner   boundaries
of $ {\rm graph} (u_1)$ and $ {\rm graph} (\widetilde u_1)$.

\item  Finally,  given $j\in\mathbb N$ if we let the
constant $G_\tau$ depend on $j$ as well, then $M$ contains $j$ disjoint
3-valued graphs and the description in the previous
paragraph holds for each of them.
\end{itemize}
\end{theorem}

\subsection{The final step in the proof of
the Extrinsic Curvature Estimate for Planar Domains.} \label{number2}
 Note that Theorem~\ref{mainextension}
is also true for minimal disks, and follows
from  the extension results of minimal multi-valued graphs by Colding
and Minicozzi~\cite{cm21},
which  motivated our work. However,
the following curvature estimates depend on
the nonzero value of the constant mean curvature and are not true for
minimal surfaces. For the reader's convenience,
we  recall the statement of Theorem~\ref{thm3.1}.\\

\begin{theorem}[see Theorem~\ref{thm3.1}]
Given $\ve>0$, $m\in \N$  and $H \in(0, \frac{1}{2\ve})$, there exists
a constant $K(m,\ve, H)$ such that the following holds.  Let
$M\subset \ov{\B}(  \ve  )$ be a compact, connected $H$-surface of genus zero
with at most $m$ boundary components, $\vec{0}\in M$, $\partial M \subset \partial
\B( \ve )$ and $M$ has zero flux. Then:
$$|A_M |(\vec{0})\leq   K(m,\ve, H).$$
\end{theorem}

\begin{proof}
 Arguing by contradiction, suppose that the theorem fails. In this
case, for some $\ve \in (0,\frac{1}{2})$, there exists a sequence
$M_{n}$ of $H$-surfaces satisfying the
hypotheses of the theorem and $|A_{M_{n}}|(\vec{0})>n$. After replacing
$M_{n}$ with a subsequence and applying a small translation,
that we shall still call $M_{n}$,
Theorems~\ref{extmvg2} and  \ref{mainextension} show that after
 composing by a fixed rotation,
given any $k\in \N$, there exists an $n(k)\in \N$, such that
for $n>n(k)$, there exist  $2k$  pairwise disjoint
3-valued graphs
$G^{down}_{n,1},G^{up}_{n,1},  \ldots, G^{down}_{n,k},G^{up}_{n,k}$  in $M_n$
on a fixed horizontal scale, i.e.,
they are all 3-valued graphs over a fixed annulus $A$
in the $(x_1,x_2)$-plane,  $G^{down}_{n,j}\cap W[G^{up}_{n,j}]\neq \O$
for $j\in \{1,\ldots,k\}$, and the  gradients of the graphing functions are
bounded in norm by 1; here the superscripts ``up" and ``down" refer to
the pointing directions of the unit normals to the graphs.

Thus, as $n$ goes
to infinity,  at least two of these
disjoint 3-valued graphs with constant mean curvature $H$ over $A$ that have
mean curvatures vectors pointing toward each
other, are becoming arbitrarily close to
each other.  This situation  violates
the maximum principle. Alternatively to obtain
a contradiction, note that as the number $k$ of
these pairwise disjoint graphs goes to infinity,
there exists a sequence $\{G^{up}_{n,j(n)},G^{down}_{n,j(n)}\}$ of
 associated  pairs of
oppositely oriented 3-valued graphs that collapses smoothly
to an annulus of constant mean curvature $H$
that is a graph over $A$ and whose
nonzero mean curvature vector points upward
and downward at the same time, which is impossible.
This contradiction proves that the norm of
the second fundamental form of $M$ at the origin must have a uniform bound.
\end{proof}

\section{The Extrinsic Curvature and
Radius Estimates for $H$-disks.} \label{number3}

In this section we prove  extrinsic curvature
and radius estimates for $H$-disks. The extrinsic curvature estimate
will be a simple consequence of the Extrinsic Curvature Estimate for
Planar Domains given in Theorem~\ref{thm3.1},
once we prove that an $H$-disk with $H\leq 1$ and with boundary contained outside
 $\ov{\B}(R)$, where $R\leq \frac12$, cannot intersect  $\B(R)$ in a component with an
arbitrarily large number of boundary components; see the next proposition
for the existence of this bound.

\begin{proposition} \label{number} There exists $N_0\in \N$ such that
for any $R\leq{\frac{1}{2}}$ and $H\leq 1$, if $M$ is a compact disk
of constant mean curvature
$H$ (possibly $H=0$) with $\partial M\subset [\rth-\B(R)]$
and $M$ is transverse to $\partial \B(R)$, then each component of
$M\cap \B(R)$ has at most $N_0$ boundary components.

Furthermore, there exists  an $\overline{R}\in(0,\frac12)$ such that
whenever $R\leq \overline{R}$, then each
component of $M\cap \B(R)$ has at most $5$ boundary components.
\end{proposition}
\begin{proof}
Since the surface $M$ is transverse to the sphere $\partial \B(R)$,
there exists a $\de\in (0,\frac{R}4)$ such that $M$ intersects the closed $\delta$-neighborhood
of $\partial \B(R)$ in components that are smooth compact annuli, where each such component
has one boundary curve
in $\partial \B(R+\delta)$ and one boundary curve
in $\partial \B(R-\delta)$,  and such that each of the spheres
$\partial \B(R+t)$ intersects $M$ transversely for $t\in [-\de,\de]$. For the remainder of
this proof we fix this value of $\de$, which depends on $M$.

If $M$ is minimal, then the convex hull property implies $N_0$ can be taken to be $1$.
Assume now that $M$ has constant mean curvature $H$,  $H\in(0,1]$.

Let $\Sigma$ be a component of $M\cap \overline{\B}(R)$ with boundary curves
$\Delta=\{\beta, \beta_1, \beta_2, \ldots, \beta_n\}$.  Here $\beta$ denotes
the boundary curve of $\Sigma$ which is the boundary of the annular component
of $M-\Sigma$, or equivalently, $\beta$ is one of the two boundary curves of
the component of $M-\Sigma$ that has $\partial M$
in its boundary; see Figure~\ref{figure1pdf2}.

\begin{figure}
\begin{center}
\includegraphics[width=6.1cm]{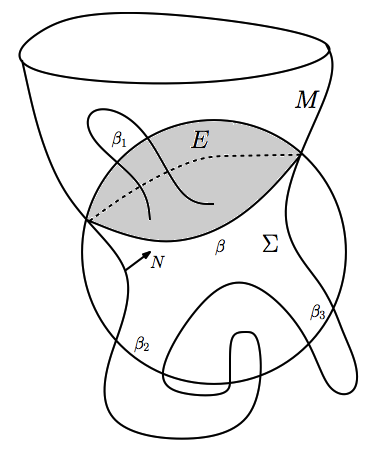}
\caption{A possible picture of $M$, $\S$ and the disks $D_{\be}$ and $E$.
 Here the component $\S$ of $M\cap \B(R)$ has  4 boundary components $\be, \be_1, \be_2, \be_3$
and $\Int(E)\cap M$ contains the simple closed curve $\be_1$.}
\label{figure1pdf2}
\end{center}
\end{figure}

Let $E$ be one of the two closed disks in $\partial \B(R)$ with boundary $\beta$.
Let $D_\beta$
be the open disk in $M$ with boundary $\beta$ and note that
$\Sigma\subset D_\beta$.  Next consider the piecewise-smooth
immersed sphere $D_\beta\cup E$ in $\rth$ and suppose that $D_\beta\cap E$ is a collection of $k$
simple closed curves. Then, after applying $k$
surgeries to this sphere in the open $\delta$-neighborhood of $E$, we obtain a
collection of $(k+1)$ pairwise disjoint piecewise-smooth embedded
spheres; for the after surgery picture when $k=1$, see Figure~\ref{fig2pdf}.
We can assume that each of these pairwise  disjoint piecewise-smooth embedded
spheres  contains a smooth compact connected subdomain  in $D_{\be}$ so that its complement in the
piecewise-smooth sphere
consists of a finite number of disks contained in  spheres of the form $\partial \B(R+t)$ for $t\in (-\de,\de)$.

\begin{figure}
\begin{center}
\includegraphics[width=5.7cm]{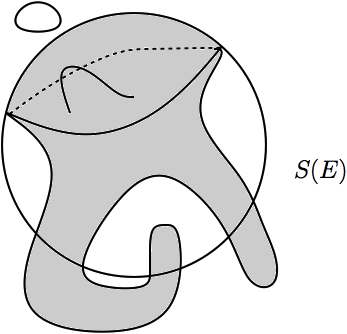}
\caption{The picture of the
embedded piecewise-smooth sphere $S(E)$ which is
the boundary of the
piecewise-smooth ball
$B(E)$; $S(E)$ is a sphere formed by first applying repeated surgeries to $D_\be$ along
the disk $E\subset \partial B(R)$
and then attaching $E$ to  the remaining disk with boundary $\be$.}
\label{fig2pdf}
\end{center}
\end{figure}

 Let $S(E)$ denote the sphere
in this collection which contains $E$, let $\Omega(E)$ be the smooth compact connected
subdomain $[D_\be\cap S(E)]-\Int(E)$ and let $B(E)$ denote the closed ball in
$\rth$ with boundary $S(E)$; see Figures~\ref{figure1pdf2} and \ref{fig2pdf}.

\begin{assertion} \label{ass:N} Let $\Gamma$ be the
subcollection of curves in $\Delta$ which
are not contained in   $E$.
The number of elements in $\G$ is bounded independently of $R\leq\frac12$ and $H\in(0,1]$ and
the choice of the component $\Sigma$. Furthermore, there exists an
$\overline{R}>0$ such that if $R\leq \overline{R}$, then $\G$ has at most two elements.
\end{assertion}
\begin{proof} Assume  that
$\G\neq \O$ and we shall prove the existence of the desired bounds on the number of elements in $\G$.

The condition   $\G\neq \O$  implies there are points of $\Omega(E)$ of distance greater than
$R+\de$ from the origin, and so there is a point $p\in \Int(\Omega(E))$ that is  furthest
from the origin. Since $S(E)-\Omega(E)$ is contained in $\B(R+\de)$,  $p$ is also a point of $B(E)$
that is furthest from the origin.   Since $B(E)\subset \ov{\B}(|p|)$,  the mean curvature comparison principle
implies that the mean curvature vector of $\Omega(E)$ points into $B(E)$ at $p$
and that it also points towards the origin. Since $\Omega(E)$ is a connected surface of
positive mean curvature, it follows that all of the
 mean curvature vectors of $\Omega(E)\subset \partial B(E)$
point into $B(E)$.

Each $\a\in \G$ bounds an open  disk $D_{\a}\subset D_\beta$
which initially enters   $\rth-\ov{\B}(R)$ near $\a$ and  these disks form a
pairwise disjoint collection. Each of these disks $D_{\a}$ intersects the $\delta$-neighborhood
of $\partial \B(R)$ in a subset that contains a component that is compact annulus with one boundary curve
 $\a$ and a second boundary component $\g(\a,\delta)$ in $\partial \B(R+\delta)$.
Since $H\leq 1$ and $R\leq\frac12$,  the mean curvature comparison principle shows that the component
$\widehat{D}_{\a}$ of $S(E)\cap[\rth-\B(R+\delta)]$ containing $\g(\a, \delta)$,  must
contain a point $p_{\a}$ with $|p_{\a}|\geq \frac{1}{H}\geq 1$ of maximal distance from the
origin.
Also note that the mean curvature vector of $\widehat{D}_{\a}$ points
towards the origin at $p_{\a}$. Hence,
since $B(E)$ is mean convex along $\Omega(E)$,
points in $B(E)$ sufficiently close to $p_\a$ are contained in $\ov{\B}(p_a,|p_\a|)$.
Once and for all, we make  for each $\a\in \G$, a particular choice
for   $p_{\a}$ if there is more than one possible choice.


Since $\rth-\ov{\B}(R+\delta)$ is
simply-connected, elementary separation properties imply that for each $\a\in \G$,
$\rth-[{\widehat{D}}_{\a}\cup \ov{\B}(R+\delta)]$ contains two components, and
let $B_\a$ be the closure of the bounded component; see Figure~\ref{fig6pdf}.

\begin{figure}
\begin{center}
\includegraphics[width=4.2cm]{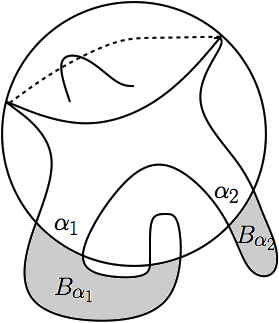}
\caption{Picture depicting the pairwise disjoint balls $B_\a$, $\a \in \G$, in the case
where $\G$ contains  two components $\a_1=\be_2, \a_2=\be_3$, when $\de$ is assumed to be 0.}
\label{fig6pdf} 
\end{center}
\end{figure}

 By construction
$\partial B_\a$ is an embedded piecewise-smooth compact surface,
the domain $\widehat{D}_{\a}\subset \partial B_\a$ is connected and
$\partial B_\a-\widehat{D}_{\a}$ consists of a finite number of
precompact domains in $\partial \B(R+\de)$.  Since the point $p_{\a}$
is a point of $B_\a$ of maximal distance from the origin, the
inward pointing normal of $\partial B_\a$ at  $p_{\a}$ points toward
the origin. Therefore  $\widehat{D}_{\a}$ has positive mean curvature
as part of the boundary of $B_\a$ oriented by its inward pointing normal.

We claim    that  $\{B_\a\}_{\a\in \G}$ is a
collection of indexed compact pairwise disjoint domains.
To prove this claim it suffices to show
that $S(E)\cap B_\a=\widehat{D}_{\a}$. Clearly, since for each $\a\in \G$,
$\widehat{D}_{\a}\subset S(E)$ and $\widehat{D}_{\a}\subset \partial B_\a$, then
$\widehat{D}_{\a}\subset S(E)\cap B_\a$. If $S(E)\cap B_\a\not \subset\widehat{D}_{\a}$,
then there exists  a point $p\in [S(E)-\widehat{D}_{\a}]\cap B_\a$ that is
furthest from the origin. By our earlier small positive choice of $\de>0$, the sphere
$\partial \B(R+\de)$ intersects $M$ transversely and so the point  $p$
lies in the interior of $B_\a$. Let $r_p$ be the ray $\{tp\mid t\geq 1 \}$,
let $t_{0}>1$ be the smallest $t>1$ such that $t_p\in\widehat{D}_{\a}$
and let $\gamma$ be the open segment $\{tp \mid 1<t<t_{0}\}$.
Since $p\in [S(E)-\widehat{D}_{\a}]\cap B_\a$  is a point in this set  that is furthest from
the origin,  the closed segment $\overline{\g}$ intersects $S(E)$ only
at its end points, namely $p$ and $t_{0}p$, which means that
the segment $\gamma$  is either in $B(E)$ or it is contained in the
complement of $B(E)$.

 On the one hand, since $ [D_{\beta}\cap S(E)]\subset \partial B(E)$
has positive mean curvature as part of the boundary of $B(E)$ oriented by its
inward pointing normal and   the mean curvature vector of  $D_\beta$ at $p$
points towards the origin, the ray $\g$ enters the complement of $B(E)$ near
$p$, which implies that $\gamma \subset [\rth - B(E)]$.  On the other hand,
$\wh{D}_\a$ is mean convex as part of the boundary of  $B(E)$ and also
as part of the boundary of
$B_\a$. Since $p$ lies in the interior of $B_\a$, then
$\g \subset B_\a$, which implies that  $\g$ is contained in the interior of
$B(E)$ near $t_{0}p \in \wh{D}_\a$.
This is a contradiction which proves  that  $S(E)\cap B_\a=\widehat{D}_{\a}$.
Hence,   $\{B_\a\}_{\a\in \G}$ is a
collection of compact pairwise disjoint domains.

Although we did not subscript the collection $\{B_\a\}_{\a \in\G}$ with
the variable $\delta$, the domains in it do depend on $\delta$. Letting
$\delta \to 0$, we obtain a related collection of limit compact domains
$\{B_\a\}_{\a\in \G}$, which we denote in the same way and which are pairwise
disjoint. Let $r_{\a}$ be the ray $\{s\frac{p_{\a}}{|p_{\a}|}\mid s>0\}$ and for $t\in (0,|p_{\a}|]$, let
$\Pi(\alpha)_t$ be the plane perpendicular to $r_\a$ at the point $t\frac{p_{\a}}{|p_{\a}|}$.
A standard application of the Alexandrov reflection principle to the region $B_\a$
and using the family of
planes $\Pi(\alpha)_t$, gives that the connected component $U_\a$ of
$\widehat{D}_{\a}- \Pi(\alpha)_{\frac{R+|p_\a|}{2}}$
containing $p_\a$ is graphical over its projection to $\Pi(\alpha)_{\frac{R+|p_\a|}{2}}$
and its  image $\widehat{U}_\a$, under reflection in
$\Pi(\alpha)_{\frac{R+|p_\a|}{2}}$, is contained in $B_\a$.
Thus, if $\alpha_1,\alpha_2 \in \G$ and $\a_1\neq \a_2$,
then $\widehat{U}_{\a_1}\cap \widehat{U}_{\a_2}=\O$ because $B_{\a_1}\cap B_{\a_2}=\O$.

Since $R\leq \frac12$ and $|p_\a|\geq \frac1H \geq 1$, the point $p_\a$
has height at least $\frac{1-R}{2}\geq \frac{1}{4}$ over the plane
$\Pi(\alpha)_{\frac{R+|p_\a|}{2}}$.  Let $\wh{p}_{\a}\in \widehat{U}_\a$
be the point in $\partial \B(R)\cap \partial B_\a$ that is the reflection of $p_{\a}$ in
the plane $\Pi(\alpha)_{\frac{R+|p_\a|}{2}}$. By the uniform curvature
estimates in~\cite{rst1} for oriented  graphs with constant mean
curvature (graphs are stable with curvature estimates away from their boundaries), it follows
that each of the graphs $\widehat{U}_\a$ contains  a disk $\wh{F}(\a) $
that is a radial graph over a closed geodesic disk $D(\wh{p}_{\a},\ve R)$ in
$\partial \B(R)$ centered at $\wh{p}_{\a}$ and of fixed geodesic radius
$\ve R>0$, where ${\ve}$ is independent of $M$, $R$, $\alpha$ and $H\in (0,1]$.
Since the surfaces $\widehat{U}_\a$ form a pairwise disjoint collection
of surfaces,  the distances on $\partial \B(R)$
between the centers of  different disks of the form
$D(\wh{p}_{\a},\ve R)$, $\a\in \G$, must be greater than ${\ve R}{}$. Therefore,
 $\{D(\wh{p}_{\a},\frac{\ve R}{2}) \mid \a \in \G\}$ is a pairwise
 disjoint collection of disks in $\partial \B(R)$.  Since a sphere
 of radius $R$ contains a uniformly bounded number of pairwise
 disjoint geodesic disks of radius  $\frac{\ve R}{2}$, independent of $R$, the last
 observation implies the first statement in Assertion~\ref{ass:N}.

We now prove the second statement in the assertion.  Arguing by contradiction,
suppose  $D(n)$ is a sequence of disks satisfying the conditions of the proposition,
the radii  of these disks satisfy $R_n\to 0$ and  $\Sigma_{n} \subset D(n)$ is a
sequence of components with at least three boundary components
$\{\a_1,\a_2,\a_3 \}$ in $\G(n)$, where $\G(n)$ are components
of $\Delta(n)=\partial \Sigma_{n}$ that are not contained in the
closed disk $E(n)$. Now replace the disks $D(n)$  by the scaled
disks $\frac{1}{2R_n}D(n)$ with mean curvatures $H_n$ converging
to 0 as $n\to \infty$. For $k=1,2,3$, a subsequence of the related
sequence of stable constant mean curvature graphs $\widehat{U}_{\a_k}^n$ defined
earlier converges to a flat plane $\Pi_k$ tangent to $\partial \B(\frac{1}{2})$.
Since the graphs $\widehat{U}_{\a_1}^n$, $\widehat{U}_{\a_2}^n$,
$\widehat{U}_{\a_3}^n$ are pairwise disjoint, if  $R_n$ is sufficiently
small, the number of these graphs must be at most two, otherwise
after choosing a subsequence, two
of the related limit planes  $\Pi_1$, $\Pi_2$, $\Pi_3$ must coincide
and in this case one would find that the two related graphs  $\widehat{U}_{\a_1}^n$, $\widehat{U}_{\a_2}^n$,
$\widehat{U}_{\a_3}^n$ intersect for $n$ sufficiently
large. This contradiction completes the proof of Assertion~\ref{ass:N}.
\end{proof}

Proposition~\ref{number} follows immediately from the estimates in Assertion~\ref{ass:N}.
\end{proof}

In the next lemma we prove curvature estimates for
$H$-disks that depends on the nonzero value $H$ of the mean
curvature.

\begin{lemma} \label{lem:excest} Given $\delta>0$ and $H\in (0,\frac 1{2\delta})$,
there exists a constant $K_0(\delta, H)$ such that for any $H$-disk
${\mathcal D}$,
{\large$${\large \sup}_{\large \{p\in {\mathcal D} \; \mid \;
d_{\rth}(p,\partial {\mathcal D})\geq  \delta \}}
|A_{\mathcal D}|\leq K_0(\delta,H).$$}
\end{lemma}

\begin{proof}
  After translating ${\mathcal D}$, we may assume
  that $p=\vec{0}$ and that $\partial \B(\delta)$ intersects
  $\mathcal D$ transversally.  By
  Proposition~\ref{number}, there is a universal  $N_0\in \N$
 such that the component $M$ of
${\mathcal D}\cap \B(\delta)$ containing $\vec{0}$
has at most $N_0$ boundary components.

Since $M\subset {\mathcal D}$, the planar domain $M$ has zero flux.
After setting $\ve=\delta$ and applying Theorem~\ref{thm3.1}
 to $M$,
we find that there is a constant $K_0(\delta, H)$
such that  $|A_M| (\vec{0})\leq K_0(\delta,H)$, which proves the lemma.
\end{proof}

\subsection{The Extrinsic Radius Estimate}

The next theorem states that there exists
an upper bound for the extrinsic distance from a
point in an $H$-disk to its boundary.

\begin{theorem}[Extrinsic Radius Estimates] \label{ext:rest}
There exists a constant ${\mathcal R}_{0}\geq \pi$ such that any $H$-disk
$\cD$ has extrinsic radius less than {\Large $\frac{{\mathcal R}_0}{H}$.}
In other words, for any point $p\in \cD$,
$$d_{\rth}(p,\partial \cD)<{\cR_0}/{H}.$$
\end{theorem}
\begin{proof}
By scaling arguments, it suffices to prove the theorem for $H=1$.
Arguing by contradiction, suppose that the radius estimate fails.
In this case, there exists a sequence
of $1$-disks $\cD_n$ passing through  the origin such that for each $n$,
$d_{\rth}(\vec{0},\partial \cD_n)\geq n+1$. Without loss of generality,
we may assume that $\partial \B(n)$ intersects
$\cD_n$ transversally. Let $\Delta_n$ be the smooth
component of $\cD_n \cap \B(n)$ with $\vec{0}\in \Delta_n$. By
Lemma~\ref{lem:excest}, the surfaces $\Delta_n$
have uniformly bounded norm of the second fundamental form.
A standard compactness argument, see for instance  Section~\ref{sec3} in this manuscript or the paper~\cite{mt4},
gives that a subsequence of $\Delta_n$ converges with multiplicity one to a
genus zero, strongly Alexandrov embedded\footnote{$\Delta_\infty$
is the boundary of a properly immersed complete
three-manifold $f\colon N^3\to \rth$ such that $f|_{\Int(N^3)}$ is
injective and $f(N^3)$ lies on the mean convex side of $\Delta_{\infty}$.} 1-surface
$\Delta_\infty$ with bounded norm of the second fundamental form.

By the Minimal Element Theorem in~\cite{mt4}, for some divergent sequence of
points $q_n\in \Delta_{\infty}$, the translated surfaces $\Delta_{\infty}-q_n$
converge with multiplicity one  to a strongly Alexandrov embedded surface
$\wt{\Delta}_\infty$ in $\rth$ such that the component passing
through $\vec{0}$ is an embedded Delaunay surface. Since a Delaunay surface
has nonzero flux, we conclude that the original disks $\cD_n$ also have
nonzero flux for $n$ large, which is a contradiction. This contradiction
proves that the extrinsic radius of a 1-disk $\cD$ is bounded by a universal constant, and
Theorem~\ref{ext:rest} now follows.
\end{proof}

\subsection{The Extrinsic Curvature Estimate}
In this section we prove the extrinsic curvature estimate stated in the Introduction.

\begin{proof}[Proof of Theorem~\ref{ext:cest}]
Arguing by contradiction, suppose that the theorem fails for some $\de, \cH>0$.
In this case there exists a sequence of $H_n$-disks
with $H_n\geq\cH$ and points $p_n\in \cD_n$ satisfying:
\begin{equation}\label{eq:1} \de\leq d_{\rth}(p_n, \partial \cD_n), \end{equation}
\begin{equation} \label{eq:2} n\leq |A_{\cD_n}|(p_n) .
\end{equation}

Rescale these disks by $H_n$ to obtain the sequence of 1-disks
$\wh{\cD}_n=H_n\cD_n$ and a related sequence of points $\wh{p}_n=H_n p_n$.
By definition of these disks and points, and  equations~\eqref{eq:1} and
\eqref{eq:2}, and Theorem~\ref{ext:rest}, we have that
\begin{equation}\label{eq:3}
\de \mathcal H\leq \de H_n \leq d_{\rth}(\wh{p}_n, \partial \wh{\cD}_n)\leq\mathcal R_0,
\end{equation}  \begin{equation}\label{eq:4}
\frac{n}{H_n}\leq |A_{\wh{\cD}_n}|(\wh{p}_n) .
  \end{equation}

Equation~\eqref{eq:3} and Lemma~\ref{lem:excest} imply that
\[
|A_{\wh{\cD}_n}|(\wh{p}_n)\leq K_0(\delta\mathcal H,1).
\]
 This inequality, together with equations~\eqref{eq:3} and~\eqref{eq:4}, then gives
\[
\frac{\de}{\cR_0}n \leq\frac{n}{H_n}
\leq|A_{\wh{\cD}_n}|(\wh{p}_n)\leq K_0(\de\mathcal H, 1),
\]
which gives a contradiction for $n$ chosen sufficiently large.  This
contradiction proves  the desired curvature estimate.
\end{proof}

\vspace{.3cm}
\center{William H. Meeks, III at profmeeks@gmail.com\\
Mathematics Department, University of Massachusetts, Amherst, MA 01003}
\center{Giuseppe Tinaglia at giuseppe.tinaglia@kcl.ac.uk\\ Department of
Mathematics, King's College London,
London, WC2R 2LS, U.K.}

\bibliographystyle{plain}
\bibliography{bill}

\end{document}